\documentclass{article}
\usepackage{geometry}
\newcommand{\ak}[1]{\textcolor{black}{#1}}
\newcommand{\rev}[1]{\textcolor{black}{#1}}
\usepackage{graphicx}
\usepackage[dvipsnames]{xcolor}
\usepackage{algorithm}
\usepackage{amsthm}

\usepackage{amsmath, amsfonts, bm, amssymb, enumitem, tabularx}
\usepackage{amsfonts, graphicx, grffile,fullpage, float, latexsym, mathtools, bbm}
\usepackage{hyperref}
\hypersetup{
    colorlinks=true, 
    citecolor=black,
     linkcolor=blue,
     urlcolor=blue,
    linktoc=all,     
    linkcolor=black,  
}
\usepackage{cleveref}
\usepackage{multicol}
\usepackage{tikz}
\usepackage{subcaption}

\usepackage[sort, numbers, compress]{natbib}

\makeatletter
\def\thanks#1{\protected@xdef\@thanks{\@thanks
        \protect\footnotetext{#1}}}
\makeatother

\usepackage{Shorthands}

\begin{document}

\title{Closed-loop Analysis of ADMM-based\\ Suboptimal Linear Model Predictive  Control}

\date{}
\author{Anusha Srikanthan{$^{*\dagger}$}, Aren Karapetyan{$^{*\ddagger}$}, Vijay Kumar{$^\dagger$}, Nikolai Matni{$^\dagger$}\thanks{{*-Equal contribution.}}\thanks{This research is supported by the Swiss National Science Foundation under NCCR Automation (grant agreement 51NF40\_180545), as well as by NSF awards CPS-2038873, SLES-2331880, and NSF CAREER award ECCS-2045834.} \\
{$^\dagger$}University of Pennsylvania, PA, USA\\
{$^\ddagger$}ETH Zürich, Switzerland\\
\tt\small \{sanusha, kumar, nmatni\}@seas.upenn.edu\\
\tt\small akarapetyan@ethz.ch%
}

\maketitle

\begin{abstract}
Many practical applications of optimal control are subject to real-time computational constraints.  When applying model predictive control (MPC) in these settings, respecting timing constraints is achieved by limiting the number of iterations of the optimization algorithm used to compute control actions at each time step, resulting in so-called suboptimal MPC.  This paper proposes a suboptimal MPC scheme based on the alternating direction method of multipliers (ADMM).  With a focus on the linear quadratic regulator problem with state and input constraints, we show how ADMM can be used to split the MPC problem into iterative updates of an unconstrained optimal control problem (with an analytical solution), and a dynamics-free feasibility step. We show that using a warm-start approach combined with enough iterations per time-step, yields an ADMM-based suboptimal MPC scheme which asymptotically stabilizes the system and maintains recursive feasibility. 
\end{abstract}


\section{Introduction}
\label{sec:introduction}
The applicability of model predictive control (MPC) to time and safety critical  applications~\cite{farina2018hierarchical, minniti2021adaptive, sun2022comparative, rosolia2022unified} is often limited by its high computational cost. 
To implement MPC, at each time step a constrained optimal control problem (OCP) \ak{needs to} be solved over a fixed horizon. However, in practice, in the presence of computational limits, OCPs are solved approximately using numerical solvers running iterative algorithms up to a given accuracy \cite{liao2020time, zanelli2021lyapunov, karapetyan2023finite, srikanthan2023augmented}. In this work, we study the stability of the resulting closed-loop \emph{system-optimizer dynamics}, in which we explicitly account for the finite number of alternating direction method of multipliers (ADMM) \cite{boyd2011distributed} steps used to solve the constrained OCP. In doing so, we extend the results in \cite{schulze2021closed} to a much wider class of problems, in particular allowing for active state and input constraints.

The closed-loop stability of suboptimal or approximate MPC schemes is studied in the context of gradient-based solvers running a fixed number of iterations in \cite{zanelli2021lyapunov, liao2020time}, wherein local asymptotic stability is shown  when sufficiently many gradient descent steps are performed per OCP solve. Later works \cite{liao2021analysis, leung2021computable} derive less stringent conditions  for the linear quadratic regulator (LQR) setting and no state constraints, with the computational time-optimality tradeoff analyzed in \cite{karapetyan2023finite}.  An alternative strategy towards handling the computational bottleneck of constrained optimization involves the use of layering~\cite{matni2016theory,matni2024towards} or splitting methods~\cite{o2013splitting, srikanthan2023augmented} where complex constraints are separated into different layers, in-line with the practical needs of  hardware implementation. The authors in~\cite{o2013splitting} propose an ADMM-based split to solve the OCP but do not consider its closed-loop properties. Such an analysis, with a fixed number of ADMM iterations, is carried out  in \cite{schulze2021closed}. However, the dynamics are only considered within a small enough region of attraction (ROA) such that no constraints activate, leading to linear closed-loop dynamics. In this work, we extend the results of \cite{schulze2021closed}, by considering the combined nonlinear system-optimizer dynamics, in the spirit of time-distributed optimization \cite{liao2020time}, and by making use of linear convergence results for the ADMM algorithm \cite{nishihara2015general}.
We also note the rich body of work~\cite{east2018admm, rey2020admm, gracia2024efficient, krupa2024sparse} that explores the use of ADMM structure to improve the computational efficiency of solving constrained OCPs for MPC, but does not analyze the closed-loop properties. Furthermore, none of the papers consider the problem of early termination of ADMM, nor the effect of the resulting suboptimality on closed-loop stability of the system. 

Inspired by the convergence results for ADMM \rev{in} static optimization problems~\cite{nishihara2015general} and by the advantages of layered control architectures \cite{matni2024towards}, we consider ADMM as an operator splitting method for solving the LQR problem with dynamics, state, and input constraints. We propose solving the MPC OCP using ADMM for a fixed number of iterations, and applying feasible inputs in a receding horizon fashion. We show, that the resulting nonlinear \rev{policy} asymptotically stabilizes the system and stays recursively feasible if enough iterations are performed per OCP solve. Our contributions are as follows
\begin{enumerate} 
    \item 
    We propose a computationally efficient approximate MPC using ADMM as an operator splitting method for constrained linear quadratic regulator (LQR) problems without terminal set constraints. 
    \item To the best of our knowledge, we present the first proof of local asymptotic stability and recursive feasibility for ADMM-based MPC with finite number of iterations \rev{considering the nonlinear system-optimizer dynamics}. 
\end{enumerate} 
\vspace{-0.2mm}
\rev{To show the above, we derive a forward invariant region of attraction (ROA) for the combined system-optimizer dynamics, inspired by \cite{leung2021computable}. As a result, we provide an explicit but   conservative estimate for the number of ADMM iterations required for asymptotic stability.} We provide simulations to \rev{show that in practice a lower number of iterations suffices,} and in particular\rev{, that our results can be leveraged} to explore computation-performance tradeoffs. We conclude the article by discussing possible extensions of the work to nonlinear systems and reference tracking.

\textbf{Notation:}
The space of $n$-dimensional real numbers is denoted by $\mathbb{R}^n$ and that of positive definite  and semidefinite matrices by $\bbS_{++}^{n}$ and $\bbS_{+}^{n}$, respectively. Given a matrix $A$, we denote its largest and smallest eigenvalues \ak{by} $\lambda^+(A)$ and $\lambda^{-}(A)$, and singular values by $\sigma^{+}(A)$ and $\sigma^{-}(A)$, respectively. The Euclidean norm of a given vector $x \in \R^n$ induced by a matrix $M \in \bbS_{+}^{n}$ is denoted by $\| x\|_M$, i.e., $\|x\|_M^2 = x^\top M x$. \rev{The Kronecker product is denoted by $\kron$}. The indicator function of a vector $x$ on a set $\calX$ is defined by
\begin{equation*}
    \mathbb{I}_{\calX}(x) := \begin{cases}
        0, & \text{if } x \in \calX,\\
        \infty, & \text{\ak{otherwise}}.
    \end{cases}
\end{equation*}

\section{Problem Formulation}
Consider a linear time-invariant system of the form
\begin{equation}\label{eq:lti-sys}
    x_{t+1} = Ax_t + Bu_t,
\end{equation}
where $A \in \mathbb{R}^{n \times n},\ B \in \mathbb{R}^{n \times m}$ are system matrices, $x_t \in \mathbb{R}^n$ is the state vector, and $u_t \in \mathbb{R}^m$ is the control input. Our objective is to design a controller to stabilize the system~\eqref{eq:lti-sys} to the origin while satisfying the constraints $u_t \in \mathcal{U}$ and $x_t \in \mathcal{X}$ for all $t \geq 0$, where $\mathcal{U} \subseteq \mathbb{R}^m$ and $\mathcal{X} \subseteq \mathbb{R}^n$.  To this end, given some cost matrices $Q \in \mathbb{S}_{++}^n$  and $R \in \mathbb{S}_{++}^m$, we formulate the following parametric optimal control problem (POCP) over a horizon $N>0$, parameterized by some initial state $\xi \in \calX$ 
\begin{equation}\label{prob:lqr}
\begin{split}
    \mu^\star(\xi): = &\argmin_{\nu_{0},\hdots,\nu_{N-1}} 
    \xi_{N}^T P \xi_{N} +\sum_{i=0}^{N-1} \xi_{i}^T Q \xi_{i} + \nu_{i}^T R \nu_{i} \\
      &\qquad\text{s.t. }  \xi_{i+1} = A\xi_{i} + B\nu_{i}, \; \forall i= 0, \hdots, N-1  \\
     &\qquad \hphantom{\text{s.t. }}\xi_{i} \in \mathcal{X},\nu_{i} \in \mathcal{U}, \; \forall i= 0, \hdots, N-1\\
     &\qquad \hphantom{\text{s.t. }}\xi_{0} = \xi,\ \xi_{N} \in \calX, 
\end{split}
\end{equation}
where $\xi_{i} \in \mathbb{R}^n, \nu_{i}\in \mathbb{R}^m$, and $P$ solves the discrete-time algebraic Riccati equation
\begin{equation}\label{eq:dare}
    P = Q + A^T P A - K^TB^TPA,
\end{equation} 
with $K = (R + B^T P B)^{-1} B^T P A$. The solution  to \eqref{prob:lqr}, $\mu^\star : \mathcal{X} \rightarrow \mathbb{R}^{Nm}$, maps the initial state to the solution vector. 
We consider an MPC formulation, where at each time $t$, the  controller measures the current state, solves \eqref{prob:lqr} and applies the first input, resulting in the following closed-loop dynamics
\begin{equation}\label{eq:optimal_closed_loop}
    x^\star_{t+1} = Ax^\star_t + B\Xi \mu^\star(x^\star_t)=: f(x^\star_t),
\end{equation}
where $\Xi = \begin{bmatrix}
    I_{m} & 0_{m(N-1)}
\end{bmatrix} \in \R^{mN}$ is a selector matrix, choosing the first input.
 
 We make the following assumptions in order for the problem to be well-defined and convex.

\begin{assumption}\label{assum:riccati}
    The pair $(A, B)$ is stabilizable, the cost matrices satisfy $R \in \mathbb{S}_{++}^m$, and $Q \in \mathbb{S}_{++}^n$ (hence $(A, Q^{1/2})$ is detectable).
\end{assumption}

\begin{assumption}\label{assum:closed-mpc}
    The sets $\mathcal{X}$ and $\mathcal{U}$ \rev{are closed polytopes that contain the origin}.
\end{assumption}

Since problem~\eqref{prob:lqr} is often expensive to solve exactly, we consider its approximate solution with a splitting method as we detail in the next section. To this aim, we reformulate the POCP \eqref{prob:lqr} in the following equivalent form 
\begin{equation}\label{prob:condensed-lqr}
    \begin{array}{rl}
        V_N^\star(\xi) = \min\limits_{\boldsymbol{u,r}} &J( \boldsymbol{u}) + g(\xi, \boldsymbol{r}) \\
        \text{subject to } &\boldsymbol{u} - \boldsymbol{r} = 0, 
    \end{array}
\end{equation}
by  defining $\boldsymbol{u} = \left[\xi_0^\top, \nu_{0}^\top, \cdots, \nu_{N-1}^\top\right]^\top \in \R^{s}$, and introducing a redundant variable $\boldsymbol{r}\in \mathbb{R}^{s}$, where $s:=n+Nm$. The optimization problem~\eqref{prob:condensed-lqr} is composed of an objective defined by a sum of two functions, and a consistency constraint. The first function
 $J(\boldsymbol{u}) = \Vert \boldsymbol{u} \Vert_{M}^2$ where $M = \begin{bmatrix}
    W & G^T \\
    G & H
\end{bmatrix}$ with $W \in \bbS_{++}^n,\ G \in \R^{Nm \times n},\ H \in \bbS_{++}^{Nm}$ defined as in~\cite[Appendix A]{liao2021analysis} depend on the dynamics and cost matrices, but not on the constraints, and is a compact representation of the objective of problem~\eqref{prob:lqr}. The second function
\begin{equation*}
    g(\xi, \boldsymbol{r}) = \mathbb{I}_{\calU^N}(\boldsymbol{r}) + \mathbb{I}_{\calX^{N}}(\boldsymbol{r}) + \mathbb{I}_{F(\xi)}(\xi_0)
\end{equation*}
enforces the state and input constraints, where \rev{$ \bar{\calX} := {\calX} \times \cdots \times {\calX} \subseteq \R^{nN}$}, $\calU^N := \rev{\mathcal{X}} \times \calU \cdots \times \calU \subseteq \R^{s}$, \rev{${\calX}^N:= \{\boldsymbol{r}\in \mathbb{R}^s\ | \ \bar{A}\boldsymbol{r}\in \bar{\calX}\}$}, and  $F(\xi):= \{x\in \mathbb{R}^n \ | \ \xi-x =0\}$ and 
$$\bar{A} = \begin{bmatrix}
    \rev{A} & \rev{B} & 0 & \cdots & 0 \\
    A^2 & \rev{AB} & \rev{B} & \cdots & 0 \\
    \vdots & \vdots & \ddots & \ddots & 0 \\
    A^N & A^{N-1}B & \cdots & AB & B 
\end{bmatrix}.$$

As we show below, considering the equivalent problem~\eqref{prob:condensed-lqr} allows for a suboptimal solution of \eqref{prob:lqr} via ADMM wherein one update step solves a constraint-free LQR problem, and the other solves a dynamics-free constraint satisfaction problem.

\subsection{Suboptimal ADMM-MPC via Operator Splitting}
Solving optimization problem in~\eqref{prob:condensed-lqr} in a receding horizon fashion is computationally expensive for long horizons $N$ due to the presence of both dynamics and feasibility constraints. 

At time $t$, given an initial state $x_t \in \calX$, \rev{such that \eqref{prob:condensed-lqr} is feasible,} we consider the solution to~\eqref{prob:condensed-lqr} with over-relaxed ADMM~\cite{nishihara2015general} as follows
\begin{subequations}\label{eq:linear_ocp_admm}
    \begin{align}
        \boldsymbol{u}_t^{k+1} &\coloneqq \argmin_{\boldsymbol{u}}J( \boldsymbol{u}) + \frac{\rho}{2} \Vert \boldsymbol{u} - \boldsymbol{r}_t^k + \boldsymbol{v}_t^k \Vert_2^2,
        \label{eq:linear-ref-layer} \\
        \boldsymbol{r}_t^{k+1} &\coloneqq \argmin\limits_{\boldsymbol{r}} g(x_t, \boldsymbol{r})\!+\!\frac{\rho}{2} \Vert \alpha \boldsymbol{u}_t^{k+1}\!+\!\bar{\alpha}\boldsymbol{r}_t^k\!-\!\boldsymbol{r} + \boldsymbol{v}_t^k \Vert_2^2, 
        \label{eq:linear-control-update} \\
        \boldsymbol{v}_t^{k+1} &\coloneqq \boldsymbol{v}_t^k +  \alpha \boldsymbol{u}_t^{k+1} + (1-\alpha) \boldsymbol{r}_t^k - \boldsymbol{r}_t^{k+1},
        \label{eq:linear-dual-update}
    \end{align}
\end{subequations}
where the superscript denotes the ADMM iterations, $\rho, \alpha >0$ are tunable algorithmic parameters, $\boldsymbol{u, r}$ are primal variables, $\boldsymbol{v}$ is a scaled dual variable and $\bar{\alpha}:=1-\alpha$. The iterations start from the initial state $x_t\in \mathcal{X}$ and an initial guess  \ak{$\boldsymbol{u}_t^0\in \mathbb{R}^{s}, \boldsymbol{r}_t^0\in \mathbb{R}^{s}$, and $\boldsymbol{v}_t^0\in \mathbb{R}^{s}$}. 
We highlight here that our decomposition is such that~\eqref{eq:linear-ref-layer} has only dynamics constraints \rev{and has an explicit solution}. \rev{The step in~\eqref{eq:linear-control-update} includes only feasibility constraints and is separable across timesteps in the absence of state constraints \cite{o2013splitting}; for completeness we provide an analysis with state constraints, since one can leverage the sparse structure of the quadratic program \eqref{eq:linear-control-update} for a more efficient solution or extend the analysis to consider asynchronous ADMM updates, although we do not explore these prospects in this work.} \rev{The dual update~\eqref{eq:linear-dual-update} ensures consistency between the $\boldsymbol{u}_t^k$  and $\boldsymbol{r}_t^k$ iterates.} Several prior works~\cite{o2013splitting, srikanthan2023augmented, sindhwani2017sequential} have \rev{demonstrated} that using the ADMM algorithm to solve OCPs leads to good practical performance and empirical convergence \rev{even in the presence of state constraints. In our work, however, to leverage the linear-convergence of the ADMM algorithm~\cite{nishihara2015general} and ensure that the smoothness condition holds, we rewrite~\eqref{eq:linear-control-update} without introducing state variables.}

We therefore propose an algorithm design where, at each time step, we  solve~\eqref{prob:condensed-lqr} approximately by running $\ell$ iterations of~\eqref{eq:linear_ocp_admm} and apply the first input of the safe input vector $\boldsymbol{r}_t^\ell$. Let $\mathcal{T}$ denote the one-step  ADMM operator, i.e. the operator consisting of one pass of equations~\eqref{eq:linear_ocp_admm} such that $(\boldsymbol{r}_t^{1}, \boldsymbol{y}_t^{1}) \coloneqq \mathcal{T}(\boldsymbol{r}_t^0, \boldsymbol{y}_t^0; x_t)$ where $\boldsymbol{y}_t^{0} = \rho \boldsymbol{v}_t^{0}$ (an unscaled dual variable). The $\ell$-step operator is then defined by recursion, as
$
    \mathcal{T}^{\ell}( \boldsymbol{r}^0, \boldsymbol{y}^0; x_t) := \mathcal{T}(\mathcal{T}^{\ell-1}( \boldsymbol{r}^0, \boldsymbol{y}^0; x_t); x_t)
$, with the initialization  $\mathcal{T}^0(\boldsymbol{r}, \boldsymbol{y}; \xi) := \left[ \boldsymbol{r}^\top, \boldsymbol{y}^\top\right]^\top $ for any $\boldsymbol{r}\in\mathbb{R}^{s},\boldsymbol{v}\in\mathbb{R}^{s}$ and $\xi\in\calX$. When running the ADMM algorithm in a receding horizon fashion,
one has an estimate for the initial values of the optimization variables from the previous run. In particular, at each time $t$, one can initialize the algorithm with $(\boldsymbol{r}_{t-1}^\ell, \boldsymbol{y}_{t-1}^\ell)$. Doing so, and applying the first safe input from $\boldsymbol{r}_t^\ell$ results
in the following coupled \emph{system-optimizer closed-loop dynamics with warmstart}:
\begin{subequations}
\label{eq:suboptimal_closed-loop2}
\begin{align}
x_{t+1} &= Ax_t + \bar{B}{\boldsymbol{\varphi}}_t^\ell,\\
        \boldsymbol{\varphi}^\ell_{t+1} &= \mathcal{T}^\ell(\boldsymbol{\varphi}^\ell_{t}; x_t),
\end{align}
\end{subequations}
starting with some initialization $\boldsymbol{\varphi}^\ell_0 = \left[ \boldsymbol{r}_{0}^\top, \boldsymbol{y}_{0}^\top\right]^\top \in \mathbb{R}^{2s}$, and defining  $\bar{B}:= B \Xi_\phi$ with $\Xi_\phi :=  \begin{bmatrix}
    0_n & I_{m} & 0_{mN+s}
\end{bmatrix} $.
As the constraints are imposed in~\eqref{eq:linear-control-update}, applying control inputs from $\boldsymbol{r}_t^\ell$ results in a feasible input by construction.

\ak{It is well known \cite{boyd2011distributed, nishihara2015general}, that under Assumptions \ref{assum:riccati}-\ref{assum:closed-mpc}, if \eqref{eq:linear_ocp_admm} is run to convergence, the solution coincides with the optimal primal and dual variables  of \eqref{prob:condensed-lqr}. Given an initial state $\xi\in \mathcal{X}$, and any $\boldsymbol{\varphi}_0^\ell\in\mathbb{R}^{2s}$, the corresponding optimal solution vector is
\begin{equation}\label{eq:phi_star}
    \boldsymbol{\varphi^\star}(\xi)= \left[\boldsymbol{r}^{\star\top}(\xi), \boldsymbol{y}^{\star\top}(\xi)\right]^\top:=\lim_{\ell\rightarrow\infty}\mathcal{T}^\ell(\boldsymbol{\varphi}_0^\ell;\xi).
\end{equation}}
The following proposition from \cite{nishihara2015general} shows the linear convergence of the proposed ADMM scheme.
\begin{proposition}\label{prop:admm-convergence-thm6}\cite[Theorem 6]{nishihara2015general}
    Let Assumptions~\ref{assum:riccati}-\ref{assum:closed-mpc} hold, and suppose at time $t$, ${\boldsymbol{\varphi}}_t^\ell= \left[ \boldsymbol{r}_{t}^{\ell\top}, \boldsymbol{y}_{t}^{\ell\top}\right]^\top$ is generated by running the optimization~\eqref{eq:linear_ocp_admm} with  $\alpha \in (0, 2)$ and a step size $\rho = (pL)^{1/2} \kappa^\epsilon$ where $\kappa = L/p$, $p=2\sigma^-(M), L = 2\sigma^+(M)$ and  $\epsilon \in \R$. \rev{Then, for a sufficiently large $\kappa$, and any initial state $x_t\in \mathcal{X}$ such that \eqref{prob:condensed-lqr} 
 is feasible,} 
 the following holds with convergence rate $\tau = 1 - \frac{\alpha}{2 \kappa^{0.5 + |\epsilon|}} \in (0,1)$,
    \begin{equation*}
    \left \|{\boldsymbol{\varphi}}_t^\ell - {\boldsymbol{\varphi}}^\star(x_{t}) \right \|_{F} \leq \tau^{\ell}  \left\|{\boldsymbol{\varphi}}_{t-1}^\ell -  {\boldsymbol{\varphi}}^\star(x_{t}) \right \|_{{F}},
    \end{equation*}
    where $F = \begin{bmatrix}
    1 & 1 - \alpha \\
    1 - \alpha & 1
\end{bmatrix} \kron I_{2s}$.
\end{proposition}

The proof follows directly from \rev{\cite[Thm. 6]{nishihara2015general}}  
by noting that all assumptions therein are satisfied. Namely, $J(\boldsymbol{u}) = \|\boldsymbol{u}\|_M^2$ is $p-$strongly convex and $L$-smooth since $M\succ 0$, the second function $g$ is proper and convex, and the matrices multiplying the decision variables in the coupling constraint $\boldsymbol{u} - \boldsymbol{r}$ are both identity, and hence invertible.
\begin{remark}
    The condition on the size of $\kappa$ is equivalent to the existence of a solution to a \rev{linear matrix inequality (LMI) in \cite[Thms. 6, 7]{nishihara2015general}.} 
    Since the aspects of ADMM intrinsic design  are not the focus of this work, for the rest of the paper we assume the parameters $\alpha$ and $\rho$ are picked as in Proposition \ref{prop:admm-convergence-thm6}, and $Q$ and $R$ are such that $\kappa$ is sufficiently large. We provide specific details on how to use the LMI for finding $\alpha$ for our problem parameters in Appendix Section~\ref{app:prop1}. 
\end{remark}

We use the linear convergence of ADMM in Section~\ref{sec:closed-loop-analysis}, to show the asymptotic stability of the closed-loop dynamics \eqref{eq:suboptimal_closed-loop2} for large enough $\ell$.

\section{Optimal MPC Properties}

\ak{In this section, we outline key properties of the optimal MPC policy \eqref{prob:lqr} and the resulting closed-loop system \eqref{eq:optimal_closed_loop}.}

A key property in the proof of the stability of the suboptimal closed-loop system~\eqref{eq:suboptimal_closed-loop2} is the Lipschitz continuity of the optimal solution mapping~\eqref{prob:lqr}. To show it, consider the Lagrangian of the optimization problem~\eqref{prob:condensed-lqr} for some initial state $\xi \in \mathcal{X}$ 
\begin{equation*}
   \mathcal{L} (\boldsymbol{u}, \boldsymbol{r}, \boldsymbol{y},\boldsymbol{\lambda}; \xi) = {J}(\boldsymbol{u}) +\boldsymbol{y}^\top (\boldsymbol{u} - \boldsymbol{r}) + \boldsymbol{\lambda}^\top (\xi-\xi_0),
\end{equation*}
where $\boldsymbol{\lambda}\in \mathbb{R}^n$ and $\boldsymbol{y}\in\mathbb{R}^s$ are dual variables. Let $\boldsymbol{z}:= \left[\boldsymbol{u}^\top, \boldsymbol{r}^\top, \boldsymbol{y}^\top, \boldsymbol{\lambda}^\top\right]^\top$, then the following KKT condition is sufficient for optimality of a given $\boldsymbol{z}$
\begin{equation}\label{eq:KKT_condition}
    \nabla_{\boldsymbol{z}}\mathcal{L}(\boldsymbol{z}, \xi) + \mathcal{N}_\mathcal{Z}(\boldsymbol{z})\ni 0,
\end{equation}
where we define $\mathcal{Z}:= \mathbb{R}^{s} \times \mathcal{W} \times \mathbb{R}^{s} \times \mathbb{R}^n$, $\mathcal{W}:=\mathcal{U}^N \cap \mathcal{X}^N$ and $\mathcal{N}_{\mathcal{Z}}$ denotes the normal cone mapping for the set $\mathcal{Z}$ \rev{\cite{liao2020time}}.
We denote the set of solutions satisfying \eqref{eq:KKT_condition} by
\begin{equation} \label{eq:optimal_S}
    S(\xi):= \{\boldsymbol{z}\;|\;\nabla_{\boldsymbol{z}}\mathcal{L}(\boldsymbol{z}, \xi) + \mathcal{N}_\mathcal{Z}(\boldsymbol{z})\ni 0\}.
    \end{equation}
\rev{For the optimal solution of~\eqref{prob:condensed-lqr}, the KKT condition~\eqref{eq:optimal_S} is necessary to hold since  $J(\boldsymbol{u})$ is convex with affine constraints (e.g. \cite[Prop. 3.4.1]{bertsekas1997nonlinear}).} The necessity \rev{of~\eqref{eq:KKT_condition}}, combined with the strong regularity condition~\cite{robinson1980strongly} defined below is used to show that every optimal primal/dual solution to the POCP~\eqref{prob:condensed-lqr} is Lipschitz continuous \rev{as a function of} the parameter $\xi$.

\begin{definition}
A set-valued mapping $\Psi : \mathbb{R}^n \rightarrow \mathbb{R}^n$ is strongly regular at $x$ for $y$ if $y \in \Psi(x)$ and there exist neighborhoods $U$ of $x$ and $V$ of $y$ such that the truncated inverse mapping $\Tilde{\Psi}^{-1} : V \mapsto \Psi^{-1}(V )\cap U$ is single-valued, i.e., a function, and Lipschitz continuous on $V$.
\end{definition}

Strong regularity is equivalent to \rev{the non-singularity of the Jacobian of the mapping when $\eqref{eq:KKT_condition}$ is an equality \cite{robinson1980strongly}}.
\ak{\begin{theorem}\label{thm:strong-regularity}
    Let Assumptions~\ref{assum:riccati}-\ref{assum:closed-mpc} hold, \rev{and consider} a parameter $\xi \in \mathcal{X}$ \rev{such that \eqref{prob:condensed-lqr} is feasible. Then} the solution mapping $S$ is strongly regular for any point $\bar{\boldsymbol{z}} \in S(\xi)$. Moreover, \rev{each  bounded point $\bar{\boldsymbol{z}} \in S(\xi)$} is Lipschitz continuous for all \rev{feasible} ${\xi} \in \calX$.
\end{theorem}}

\begin{proof}
    It follows from \cite[Theorem~7]{liao2020time} that if \rev{for all $\boldsymbol{\zeta} \neq 0$ satisfying $
        \nabla_{(\boldsymbol{u},\boldsymbol{r})} h(\boldsymbol{\bar{u}},\boldsymbol{\bar{r}},\xi)
        \boldsymbol{\zeta} = 0$} 
 \begin{equation}\label{eq:sosc}
        \boldsymbol{\zeta}^\top \nabla_{(\boldsymbol{u},\boldsymbol{r})}^2 \mathcal{L}(\boldsymbol{\bar{z}}, {\xi})\boldsymbol{\zeta}>0,
    \end{equation}
\rev{where $h(\boldsymbol{\bar{u}},\boldsymbol{\bar{r}},\xi) := \begin{bmatrix}
    \boldsymbol{\bar{u}} -\boldsymbol{\bar{r}}\\
    \xi - \xi_0
\end{bmatrix}$, }then $S$ is strongly regular at any point $\bar{\boldsymbol{z}} \in S(\xi)$. Thus, showing that \rev{    
    the strong second order sufficient condition (SOSC)} \eqref{eq:sosc} is always satisfied completes the first part of the proof. The Hessian of the Lagrangian with respect to $(\boldsymbol{u}, \boldsymbol{r})$ is given by\rev{
$
    \nabla_{(\boldsymbol{u}, \boldsymbol{r})}^2 \mathcal{L}(\boldsymbol{\bar{z}}, {\xi}) = \begin{bmatrix}
        M & 0\\
        0 & 0
    \end{bmatrix}.
$}
Consider \rev{some} ${\boldsymbol{\zeta}}$ satisfying the equality in \eqref{eq:sosc}, whose  first $s$ components are $0$, then because of the coupling of the cosntraint $h$, the whole vector $\boldsymbol{\zeta}$ is $0$. Thus any $\boldsymbol{\zeta}\neq 0$ satisfying the equality must have a non-zero value in first $s$ components, thus satisfying the strong SOSC equality since $M\succ 0$. Then, the Lipschitz continuity of the finitely many solutions follows from strong regularity as proven in \cite[Theorem~3.2]{dontchev2013euler}.
\end{proof}
    
The following Corollary shows that the optimal solution mapping $\boldsymbol{\varphi}^\star$ is Lipschitz.
\begin{corollary}\label{corollary:lipschitz}
    Let Assumptions \ref{assum:riccati}-\ref{assum:closed-mpc} hold, then the solution mapping $\boldsymbol{\varphi^\star}$ \rev{defined in \eqref{eq:phi_star}} satisfies
    \begin{equation}
        \|\boldsymbol{\varphi^\star}(\xi)- \boldsymbol{\varphi^\star}(\bar{\xi})\|\leq L_1\|\xi-\bar{\xi}\|,
    \end{equation}
    \rev{for all $\xi, \bar{\xi} \in \calX$ such that \eqref{prob:condensed-lqr} is feasible, }and for some $L_1>0$.
\end{corollary}
The proof for the corollary follows directly from Theorem \ref{thm:strong-regularity}.

We now extend results from \cite{limon2006stability, liao2021analysis, leung2021computable} to show that the optimal MPC closed-loop dynamics \eqref{eq:optimal_closed_loop} are exponentially stable within a forward invariant ROA that allows for active state and input constraints. The following is a preliminary result extended from \cite{liao2021analysis}.
\begin{lemma}\label{lemma:value-fn-bound}
    Let Assumptions~\ref{assum:riccati}-\ref{assum:closed-mpc} hold. Then for all \rev{$\xi \in \mathcal{X}$ such that \eqref{prob:condensed-lqr} is feasible}, 
    the value function satisfies 
    \begin{equation*}
        \Vert \xi \Vert_{P}^2 \leq V_N^\star(\xi) \leq \delta^2 \| \xi \|^2,
    \end{equation*}
    where $\delta^2 = \lambda^+(W) + L_1 \lambda^+(H) + 2L_1 \| G \|$.
\end{lemma}

\begin{proof}
   From the definition of $V_N^\star(\xi)$
    \begin{equation*}
        V_N^\star(\xi) = \| (\xi, \mu^\star(\xi)) \|_M^2 = \| \xi \|_W^2 + 2 \langle \mu^\star(\xi), G\xi\rangle  + \| \mu^\star(\xi) \|_H^2.
    \end{equation*}
    Using the inequalities
    \begin{equation*}
        \| \xi \|_W^2 \leq \lambda^+(W) \| \xi \|^2 \text{ and } \| \mu^\star(\xi) \|_H^2 \leq \lambda^+(H) \| \mu^\star(\xi) \|^2,
    \end{equation*}
    and setting $\bar{\xi} = 0$ in Corollary~\ref{corollary:lipschitz}, it follows from the sub-multiplicativity property of norms that $\|\mu^\star(\xi) \| \leq L_1 \| \xi \|$. Using the Cauchy-Schwarz inequality, it also holds that
    \begin{equation*}
        \langle \mu^\star(\xi), G\xi\rangle  \leq \|\mu^\star(\xi) \| \| G \xi \| \leq L_1 \|G\| \| \xi \|_2^2,
    \end{equation*}
    completing the proof for the upper bound. The lower bound follows from setting the terminal cost matrix in problem \eqref{prob:condensed-lqr} to be the DARE solution $P$.
\end{proof}
To show exponential stability of the optimal closed-loop system~\eqref{eq:optimal_closed_loop}, we define $\psi(\xi):= \sqrt{V_N^\star(\xi)}$ for a given $\xi\in \mathbb{R}^n$, and use this as a Lyapunov function. Consider the set $ \Gamma = \{\xi \in \mathbb{R}^n \ |\ \| \xi \|_P^2 \leq c \}$, where $c > 0$ is small enough such that $\Gamma \subset \{\xi \in \mathcal{X}\ |\ -K\xi \in \calU\}$, \rev{, and $K= -\left(R+B^\top P B\right)^{-1}B^\top P A$}. As shown in \cite[Theorem 1]{leung2021computable}, with the choice of  $d = c\lambda^{-}(Q)/\lambda^{+}(P)$, the assumptions in \cite{limon2006stability} are satisfied such that the set
    \begin{equation}\label{eq:terminal-region}
    \Gamma_N = \{\xi \in \calX \;|\; V_N^\star(\xi) \leq Nd+c\} 
    \end{equation}
is forward invariant for the policy \eqref{prob:lqr}, and the closed-loop system \eqref{eq:optimal_closed_loop} is asymptotically stable. The following result shows that in this setting the dynamics are also exponentially stable.
\begin{lemma}\rev{\cite[Thm. 1]{leung2021computable}}\label{lemma:exp-stability}
    Let Assumptions~\ref{assum:riccati}-\ref{assum:closed-mpc} hold, and $f$ be defined as in~\eqref{eq:optimal_closed_loop}.  Then for all $\xi \in \Gamma_N$, we have that $\psi(f(\xi)) \leq \beta \psi(\xi)$, where $\beta^2 = 1 - \lambda^{-}(Q)/\delta^2 \in (0, 1)$. 
\end{lemma}
\begin{proof}
    For any $\xi\in \Gamma_N$, it holds that
    \begin{equation*}
        V_N^\star(f(\xi)) - V_N^\star(\xi) \leq -\| \xi \|_Q^2 \leq \lambda^{-}(Q) \| \xi \|_2^2.
    \end{equation*}
    Using Lemma~\ref{lemma:value-fn-bound}
    \begin{equation*}
        V_N^\star(f(\xi)) - V_N^\star(\xi) \leq -\frac{\lambda^{-}(Q)}{\delta^2} V_N^\star(\xi) \implies V_N^\star(f(\xi)) \leq \beta^2 V(\xi),
    \end{equation*}
    where $\beta^2 = 1 - \frac{\lambda^{-}(Q)}{\delta^2}$. Since $W\succeq Q$\cite{liao2021analysis}, it holds that $\lambda^{+}(W) \geq \lambda^{-}(Q)$, and hence $\beta^2 \in (0, 1)$.
\end{proof}

\ak{Next, we show the Lipschitz-continuity of $\psi$.}
\begin{lemma}\label{lemma:sqrt-value-fn-bound}
    Let Assumptions~\ref{assum:riccati}-\ref{assum:closed-mpc} hold, then
    \begin{equation*}
        |\psi(\xi) - \psi(\bar{\xi}) | \leq \delta \Vert \xi - \bar{\xi} \Vert\qquad \forall \xi, \bar{\xi} \in \Gamma_N.
    \end{equation*}
\end{lemma}
\begin{proof}
    Using the definition of $\psi$ and applying reverse triangle inequality, we have
    \begin{equation*}
        | \psi(\xi) - \psi(\bar{\xi}) |^2 = \left | \|(\xi, \mu^\star(\xi) \|_M - \| (\bar{\xi}, \mu^\star(\bar{\xi})) \|_M \right |^2
        \leq \| (\mu^\star(\xi) - \mu^\star(\bar{\xi}), \xi - \bar{\xi}) \|_M^2.
    \end{equation*}
\ak{As in the proof for Lemma~\ref{lemma:value-fn-bound}}
    \begin{align*}
        \| (\mu^\star(\xi) - \mu^\star(\bar{\xi}), \xi - \bar{\xi}) \|_M^2 &= \| \xi - \bar{\xi} \|_W^2 + 2\langle \mu^\star(\xi) - \mu^\star(\bar{\xi}), G(\xi - \bar{\xi})\rangle  + \| \mu^\star(\xi) - \mu^\star(\bar{\xi}) \|_H^2 \\
        &\leq \delta^2 \| \xi - \bar{\xi} \|^2.
    \end{align*}
    Combining the two inequalities, the result follows.
\end{proof}
\rev{The proof is similar to that of \cite[Lemma 3]{liao2020time}, and follows readily from the application of reverse triangle inequality to the square of the left hand side of the inequality.}

\section{Closed-loop Stability Analysis} \label{sec:closed-loop-analysis}

In this section we study the closed-loop properties of the system-optimizer closed-loop dynamics \eqref{eq:suboptimal_closed-loop2}, equivalently rewritten as
\begin{subequations}\label{eq:suboptimal-one-step-dyn}
\begin{align}
    x_{t+1} &=f(x_t) + \bar{B}\left({\boldsymbol{\varphi}}_t^\ell - {\boldsymbol{\varphi}}^\star(x_t)\right),\\
    \boldsymbol{\varphi}^\ell_{t+1} &= \mathcal{T}^\ell(\boldsymbol{\varphi}^\ell_{t}; x_t).
\end{align}
\end{subequations}
\rev{where $f$ is defined in \eqref{eq:optimal_closed_loop}}. 
We define $\boldsymbol{e}^{\ell}_t:={\boldsymbol{\varphi}}_t^\ell - {\boldsymbol{\varphi}}^\star(x_t)$ as a disturbance on the policy due to the suboptimality of $\ell$-step ADMM and $\bar{B}$ is defined as in \eqref{eq:suboptimal_closed-loop2}. We show the asymptotic stability of the suboptimal dynamics \eqref{eq:suboptimal-one-step-dyn}, building the proof on the results from time-distributed optimization \cite{liao2020time}, linear convergence of ADMM \cite{nishihara2015general}, and the small-gain theorem of discrete-time interconnected nonlinear systems \cite{jiang2004nonlinear}.

The following result shows the input-to-state-stability (ISS) of~\eqref{eq:optimal_closed_loop}, extended from a similar result in \cite[Theorem 3]{liao2021analysis}.

\begin{theorem}\label{thm:optimal-ISS}
    Let  Assumptions~\ref{assum:riccati} and~\ref{assum:closed-mpc} hold, then, for any $x_0\in \Gamma_N$ and inputs $\{\boldsymbol{e}^\ell_i\}_{i=0}^t \in \mathcal{E}$, the dynamics~\eqref{eq:suboptimal-one-step-dyn} satisfy
    \begin{equation*}
        \Vert x_t \Vert_P \leq \beta^t \delta \Vert x_0 \Vert + \gamma_1 \sup\limits_{t \geq 0} \Vert \bar{B}  \boldsymbol{e}^{\ell}_t\Vert,
    \end{equation*}
    where $\gamma_1 = \delta (1-\beta)^{-1}$ and $\calE = \{\boldsymbol{e}\in \mathbb{R}^{2s}\ |\ \Vert \boldsymbol{e} \Vert \leq \frac{(1-\beta)r_N}{\delta\|\bar{B}\|}\}$ and $r_N=\sqrt{Nd+c}$.
\end{theorem}

\begin{proof}
    We first show that $\Gamma_N$ is forward-invariant under the suboptimal dynamics \eqref{eq:suboptimal-one-step-dyn}. Given the dynamics~\eqref{eq:optimal_closed_loop} and~\eqref{eq:suboptimal-one-step-dyn}, it follows from Lemma~\ref{lemma:sqrt-value-fn-bound} that
    \begin{align*}
        |\psi(x^\star_{t+1}) - \psi(x_{t+1}) | &\leq \delta\Vert x^\star_{t+1} - x_{t+1}\Vert \leq \delta \| \bar{B} \boldsymbol{e}^{\ell}_t \|.
    \end{align*} 
    Let $x_t \in \Gamma_N$\rev{, then from \eqref{eq:terminal-region} it holds that $\psi(x_t) \leq r_N:=\sqrt{Nd+c}$}. Using Lemma~\ref{lemma:exp-stability}, we have
    \begin{equation*}
    \psi(x_{t+1}) \leq \psi(x^\star_{t+1}) + |\psi(x_{t+1}) - \psi(x^\star_{t+1})| \nonumber \leq \beta \psi(x_t) + \delta \Vert \bar{B} \Vert \Vert \boldsymbol{e}^{\ell}_t \Vert. 
    \end{equation*}
    Given the restriction on the signal $\boldsymbol{e}^{\ell}_t \in \calE$, it holds that $\delta \| \bar{B}\|\| \boldsymbol{e}^{\ell}_t \| \leq (1-\beta) r_N$, implying that $\psi(x_{t+1}) \leq r_N$. 
    Applying the above inequality recursively, and using the upper bound in Lemma \ref{lemma:value-fn-bound}
    \begin{equation*}
    \psi(x_t) \leq \beta^t \psi(x_0) + \delta\sum_{j=0}^{t-1} \beta^{t-1-j} \Vert\bar{B} e^{\ell}_j \Vert \leq \beta^t \delta\|x_0\| + \gamma_1\sup_{t \geq 0} \Vert \bar{B} \boldsymbol{e}^{\ell}_t \Vert.
\end{equation*}
The result follows from the lower bound in Lemma \ref{lemma:value-fn-bound}.
\end{proof}

The following theorem shows an ISS-like result for the error dynamics $\boldsymbol{e}^{\ell}_t$\rev{, akin to the time-distributed optimization results in \cite{liao2020time, liao2021analysis}}.

\begin{theorem}\label{thm:suboptimal-ISS}
    Let Assumptions~\ref{assum:riccati}-\ref{assum:closed-mpc} hold, \rev{and $x_0 \in \Gamma_N$,} then the error signal $\boldsymbol{e}^{\ell}_t$ satisfies
    \begin{equation*}
        \Vert \boldsymbol{e}^{\ell}_t \Vert_F\leq \tau^{\ell t}\| 
        \boldsymbol{e}^{\ell}_0 \|_{{F}} + \gamma_2(\ell)\sup\limits_{t \geq 0} \| \Delta x_t \|_{{F}},
    \end{equation*}
\end{theorem}
where $\Delta x_t := x_{t} - x_{t-1}$ and $\gamma_2(\ell) := L_1\rev{\|F^{\frac{1}{2}}\|}\tau^{\ell}/(1-\tau^{\ell})$. 

\begin{proof}
    Using Proposition \ref{prop:admm-convergence-thm6}, it holds that
    \begin{equation*}
          \| \boldsymbol{e}^{\ell}_t \|_F
    \leq \tau^{\ell} \left \|{\boldsymbol{{\varphi}}}_{t-1}^\ell -{\boldsymbol{{\varphi}}}^\star (x_t)
    \right \|_{{F}} \leq \tau^\ell \|\boldsymbol{e}^{\ell}_{t-1}\|_F + \tau^\ell\|{\boldsymbol{{\varphi}}}^\star (x_t) - {\boldsymbol{{\varphi}}}^\star (x_{t-1})\|_F \leq \tau^\ell \|\boldsymbol{e}^{\ell}_{t-1}\|_F + L_1\rev{\|F^{\frac{1}{2}}\|\tau^\ell\|\Delta x_t\|},
    \end{equation*}
    where the second inequality follows from the triangle inequality, and the last from Corollary \ref{corollary:lipschitz}. Applying the above recursively and using the triangle ineqality
\begin{equation*}
    \Vert \boldsymbol{e}^{\ell}_t \Vert_F \leq \tau^{\ell t} \left \| \boldsymbol{e}^{\ell}_0 \right \|_{{F}} + L_1\rev{\|F^{\frac{1}{2}}\| \tau^{\ell} \sum_{j=1}^{t} \tau^{t-j}\Vert \Delta x_j \Vert}.
\end{equation*}
Defining $\gamma_2(\ell) = L_1\rev{\|F^{\frac{1}{2}}\|}\tau^{\ell}/(1-\tau^{\ell})$, the result follows. 
\end{proof}

The following theorem summarizes the main result of the paper; it provides \rev{an explicit} lower bound on the ADMM iterations $\ell$ such that the closed-loop system \eqref{eq:suboptimal-one-step-dyn} is asymptotically stable.
\begin{theorem}\label{thm:asympt-stable-suboptimal-cl}
    Let Assumptions~\ref{assum:riccati}-\ref{assum:closed-mpc}  hold , $\tau\in(0,1)$ be as in Proposition~\ref{prop:admm-convergence-thm6}, and $\ell\geq \ell^\star$, where  
    \begin{equation*}
        \ell^\star = \max\left\{-\frac{\log{(1+\gamma_3 \gamma_1 L_1})}{\log{\tau}}, -\frac{\rev{\log}\rev{\left(\omega_1 + \omega_2 \gamma_1\|\bar{B}\|\right)}}{\log{\tau}}\right\}.
    \end{equation*} Then the set
    \begin{equation*}
    \begin{split}
        \Omega = \{(x,\boldsymbol{\varphi}) \in \Gamma_N \times \mathbb{R}^{2s}\ | \ \psi(x)\leq r_N, 
        \|\boldsymbol{{\varphi}} - \boldsymbol{{\varphi}}^\star(x)\| \leq r_e\}
        \end{split}
    \end{equation*}
    is a forward invariant ROA for the closed-loop suboptimal dynamics \eqref{eq:suboptimal-one-step-dyn} with the definitions  $r_e = \rev{\gamma_1^{-1} \|\bar{B}\|^{-1} r_N} ,\gamma_3 = 2\|F^{-\frac{1}{2}}\|\rev{\Vert}P^{-\frac{1}{2}}\|$, $\omega_1 \coloneqq \rev{\sqrt{\kappa_F}}(1+L_1 \| \bar{B} \rev{\|})$,
    and $\omega_2 \coloneqq L_1 \rev{\sqrt{\kappa_F} \|} (A - I_n) P^{-\frac{1}{2}} \| + L_1^2 \rev{\sqrt{\kappa_F}} \|\bar{B} \| \| P^{-\frac{1}{2}}\|$\rev{, where $\kappa_F = \lambda^+(F)/\lambda^-(F)$}.
    Moreover, closed-loop~\eqref{eq:suboptimal-one-step-dyn} is locally asymptotically stable within ${\Omega}$.
\end{theorem}

\begin{proof}
    First,  we show that the set $\Omega$
    is forward invariant under the dynamics \eqref{eq:suboptimal-one-step-dyn} for large enough $\ell$. Consider some $(x_t, \boldsymbol{\varphi}_t^\ell) \in \Omega$. As shown in the proof of Theorem \ref{thm:optimal-ISS}, $\psi(x_{t+1})\leq r_N$. To show that  $\|\boldsymbol{e}^{\ell}_{t+1}\|\leq r_e$, consider the one-step result from Theorem \ref{thm:suboptimal-ISS}
    \begin{align*}
         &\| \boldsymbol{e}^{\ell}_{t+1} \|_F \leq \tau^\ell \|\boldsymbol{e}^{\ell}_{t}\|_F + \tau^\ell L_1\rev{\|F^{\frac{1}{2}}\|}\|\Delta x_{t+1}\|.
         \end{align*}
    Using \eqref{eq:suboptimal-one-step-dyn}   
         \begin{align*} \rev{\|\Delta x_{t+1}\| \leq  \|(A-I_n)x_t\| + \|\bar{B}{\boldsymbol{\varphi}}^\star(x_t)\| + \|\bar{B}\boldsymbol{e}^{\ell}_t\|.}
    \end{align*}
    Combining the two, and using submultiplicativity property of norms and Corollary \ref{corollary:lipschitz}, it holds that
    \begin{align*}
        \| \boldsymbol{e}^{\ell}_{t+1} \| &\leq \tau^\ell \omega_1 \|\boldsymbol{e}^{\ell}_t\| + \tau^\ell \omega_2 \|x_t\|_P \leq \tau^\ell\left(\rev{\frac{\omega_1}{\gamma_1 \| \bar{B}\|}} + \rev{\omega_2}\right) r_N.
    \end{align*}
    Then, for $\ell\geq \ell^\star$, it holds that $\|\boldsymbol{e}^\ell_{t+1}\|\leq r_e$ showing the forward invariance of $\Omega$ and therefore $\{\boldsymbol{e}^\ell_i\}_{i=0}^t \in \mathcal{E}$. For detailed steps of the proof and derivation of constants $\omega_1$ and $\omega_2$, please refer to Appendix Section~\ref{app:thm-4}.
    To show the asymptotic stability, consider
    \begin{align*}
    \lim\limits_{t \rightarrow \infty} \sup\limits_{t \geq 0} \Vert \Delta x_t \rev{\Vert} &\leq 2\rev{\Vert} P^{-\frac{1}{2}} \Vert \lim\limits_{t \rightarrow \infty} \sup\limits_{t \geq 0} \Vert x_t \Vert_{P}\\
    &\leq 2\gamma_1\|F^{-\frac{1}{2}}\|\rev{\Vert}P^{-\frac{1}{2}} \Vert \lim\limits_{t \rightarrow \infty} \sup\limits_{t \geq 0}\|\boldsymbol{e}^\ell_t\|\rev{_F},
    \end{align*}
    where the first inequality follows from  the triangle inequality and the definition of weighted norms, and the second  from Theorem~\ref{thm:optimal-ISS}.
Then, by Theorem~\ref{thm:suboptimal-ISS}
\begin{equation*}
    \lim\limits_{t \rightarrow \infty} \sup\limits_{t \geq 0} \Vert \boldsymbol{e}^{\ell}_t \Vert_F \leq  \gamma_2(\ell) \lim\limits_{t \rightarrow \infty} \sup\limits_{t \geq 0} \Vert \Delta x_t \rev{\Vert}.
\end{equation*}
Combining the two
\begin{equation*}
    \lim\limits_{t \rightarrow \infty} \sup\limits_{t \geq 0} \Vert \boldsymbol{e}^{\ell}_t \Vert_F \leq \gamma_3 \gamma_1 \gamma_2(\ell) \lim\limits_{t \rightarrow \infty} \sup\limits_{t \geq 0} \Vert \boldsymbol{e}^{\ell}_t \Vert_F.
\end{equation*}
Using the small gain theorem \cite[Theorem 1]{jiang2004nonlinear} and the fact that $F\succ 0$, the system \eqref{eq:suboptimal_closed-loop2} is asymptotically stable if $\gamma_3 \gamma_1 \gamma_2(\ell) < 1$ which holds for $\ell\geq \ell^\star$.
\end{proof}

\rev{\begin{remark}
    Note that  the constants appearing in Theorem \ref{thm:asympt-stable-suboptimal-cl} can be computed exactly in the absence of state constraints: $\tau$ and $F$ can be estimated as detailed in \cite{nishihara2015general}, the value of $L_1$ follows by extending \cite[Corollary 2]{liao2021analysis} and the rest of the constants depend on known system parameters. With state constraints present, one can exploit the piecewise affine structure of the MPC policy~\cite{bemporad2002explicit} to derive an explicit expression for $L_1$; we do not explore such an analysis in this work. To further clarify how Theorem~\ref{thm:asympt-stable-suboptimal-cl} is used in practice, we provide a summarized explanation in Appendix Section~\ref{app:thm-4-summary}.
\end{remark}}

\section{Numerical Simulations and Discussion}\label{sec:num-sim}

\rev{To show the closed-loop behavior of the proposed ADMM-based suboptimal MPC scheme \eqref{eq:linear_ocp_admm}, we consider the following double integrator system}
\begin{equation}\label{eq:double_integrator}
    x_{t+1} = \begin{bmatrix}
        1 & 1 \\
        0 & 1
    \end{bmatrix}x_t + \begin{bmatrix}
        0 \\
        1
    \end{bmatrix} u_t,
\end{equation}
\rev{with}  $\mathcal{U} = [-0.5, 0.5]$ and  $\mathcal{X} = [-5, 5] \times [-5, 5]$. The MPC parameters are set \rev{to} $Q = I_2$, $R = 10$, and $N = 3$\rev{, and the ADMM parameters are fixed to $\alpha = 1.95$ (by bisecting values in $(0,2)$ as in \cite{nishihara2015general}), $\epsilon = 0$, and $\rho = 50$, resulting in $\tau=0.69$ and  $\kappa_F=39$. We verify that for this choice of $Q$ and $R$, $\kappa=88.76$ is sufficiently large (matrix $M$ in \cite[Thm. 7]{nishihara2015general} is positive semidefinite) and thus Proposition \ref{prop:admm-convergence-thm6} holds.  The Lipschitz constant $L_1=286.83$ is calculated as if no state constraints are present by extending the bound in \cite[Corollary 2]{liao2020time} to also include the dual variable in $\boldsymbol{\varphi}^\star$.}

\rev{The value for $\ell^\star$ from Theorem \ref{thm:asympt-stable-suboptimal-cl} is then estimated to be $85$. In practice, we note that a lower number for $\ell$ suffices for stability, as also noted in \cite{liao2020time}. The initial $\boldsymbol{\varphi_0}^\ell$ is taken to be $0$, and it is assumed to be such that the system optimizer dynamics are initialized within $\Omega$. The design of initial parameters such that the ROA initialization is satisfied is outside of the scope of this work.} In particular, we run the algorithm for a number of initial states and with $\ell = 23, 27, 30$, \rev{as shown in the state plot in Figure \subref*{fig:states-sets}. Using the ISS result in Theorem \ref{thm:optimal-ISS}, and considering a uniform-in-time, but $\ell$-dependent bound for $\|\bar{B}\boldsymbol{e}_t^\ell\|$, we also plot the corresponding robust control invariant (RCI) sets  for the dynamics \eqref{eq:suboptimal-one-step-dyn} for each $\ell$ using the tool from \cite{chen2024robust}. As expected from Theorem \ref{thm:suboptimal-ISS}, the RCI size increases with additional ADMM iterations illustrating a computation-performance tradeoff.} We also plot the input trajectories for the same values of $\ell$ and the optimal MPC in Figure~\subref*{fig:inputs}, showing that input constraints are always satisfied. To further demonstrate that our main result is a first of its kind for showing the proof of local asymptotic stability for nonlinear system optimizer dynamics, we provide additional numerical comparisons with a baseline~\cite{schulze2021closed} in Appendix Section~\ref{app:sim}.

\begin{figure}[t]
    \centering
    \subfloat[The size of the robust control invariant (RCI) set grows as the number of ADMM iterations $\ell$ is increased. The largest set is the disturbance-free maximal control invariant  set.\label{1a}]{
        \centering  
        \includegraphics[width=0.47\columnwidth]{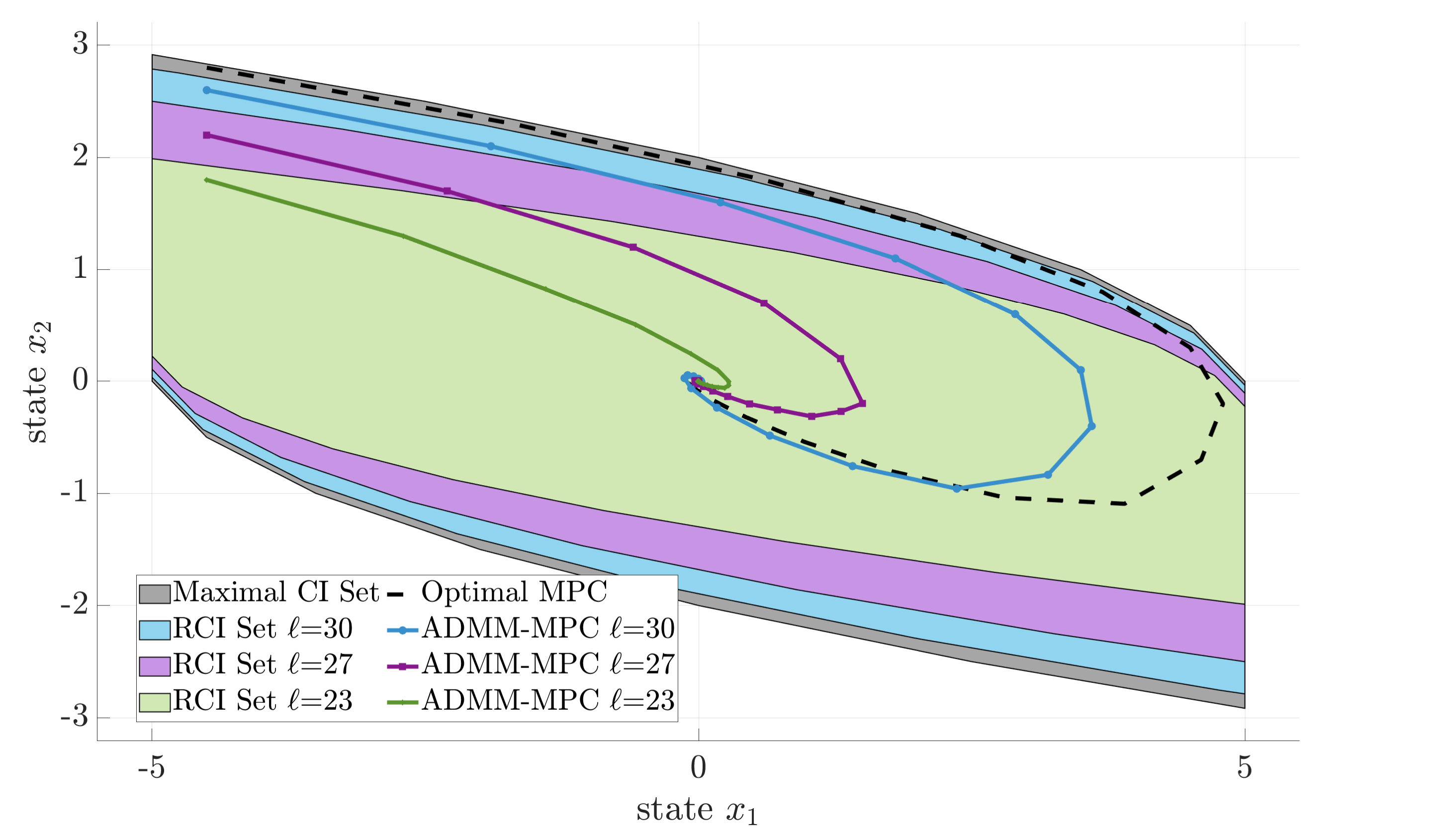}
        \label{fig:states-sets}
    } \hfill
    \subfloat[The input trajectory satisfies the constraints by construction for varying number of ADMM iterations $\ell$.\label{1b}]{
        \centering
        \includegraphics[width=0.47\columnwidth]{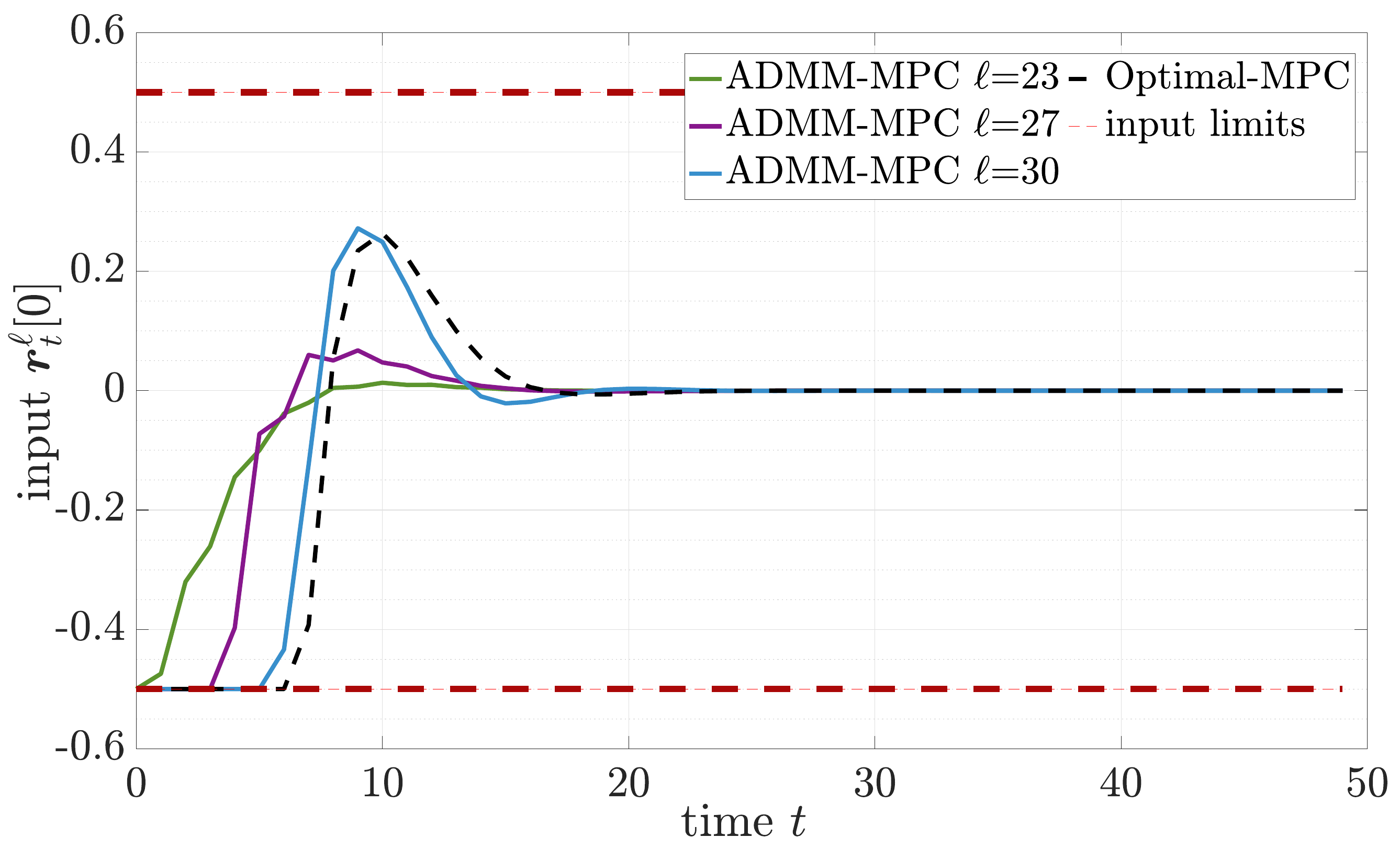}
        \label{fig:inputs}
    }
    \caption{Performance of ADMM-based suboptimal MPC for the system \eqref{eq:double_integrator} for varying number of ADMM iterations $\ell$.}    
    \label{fig:di-closed-loop}
\end{figure}

\section{Conclusions}
We study the constrained LQR problem \rev{in the context of MPC} and present a computationally efficient, ADMM-based algorithm to solve the \rev{OCP at each time step}. \rev{We derive an explicit, albeit conservative estimate for the number of ADMM iterations required for the stability of the resulting combined system-optimizer dynamics within a local ROA, and provide numerical simulations to validate the theoretical results.}
Apart from \rev{the} computational efficiency, such a split can also be of importance when layered control architectures are used in practice. The suggested split separates dynamics and feasibility constraints, which aligns with decision-making frameworks used in autonomous motion planning for robots. This offers a formal approach to addressing stability and feasibility in compute-constrained layered architectures for autonomous systems. Possible extensions include the \rev{derivation of tighter bounds on number of iterations, and the consideration of the} reference tracking problem and \rev{nonlinear system dynamics.}

\bibliographystyle{ieeetr}
\bibliography{lcsys}

\begin{thebibliography}{10}

\bibitem{farina2018hierarchical}
M.~Farina, X.~Zhang, and R.~Scattolini, ``A hierarchical multi-rate {MPC} scheme for interconnected systems,'' {\em Automatica}, vol.~90, pp.~38--46, 2018.

\bibitem{minniti2021adaptive}
M.~V. Minniti, R.~Grandia, F.~Farshidian, and M.~Hutter, ``Adaptive {CLF-MPC} with application to quadrupedal robots,'' {\em IEEE Robotics and Automation Letters}, vol.~7, no.~1, pp.~565--572, 2021.

\bibitem{sun2022comparative}
S.~Sun, A.~Romero, P.~Foehn, E.~Kaufmann, and D.~Scaramuzza, ``A comparative study of nonlinear {MPC} and differential-flatness-based control for quadrotor agile flight,'' {\em IEEE Transactions on Robotics}, vol.~38, no.~6, pp.~3357--3373, 2022.

\bibitem{rosolia2022unified}
U.~Rosolia, A.~Singletary, and A.~D. Ames, ``Unified multirate control: From low-level actuation to high-level planning,'' {\em IEEE Transactions on Automatic Control}, vol.~67, no.~12, pp.~6627--6640, 2022.

\bibitem{liao2020time}
D.~Liao-McPherson, M.~M. Nicotra, and I.~Kolmanovsky, ``Time-distributed optimization for real-time model predictive control: Stability, robustness, and constraint satisfaction,'' {\em Automatica}, vol.~117, p.~108973, 2020.

\bibitem{zanelli2021lyapunov}
A.~Zanelli, Q.~Tran-Dinh, and M.~Diehl, ``A lyapunov function for the combined system-optimizer dynamics in inexact model predictive control,'' {\em Automatica}, vol.~134, p.~109901, 2021.

\bibitem{karapetyan2023finite}
A.~Karapetyan, E.~C. Balta, A.~Iannelli, and J.~Lygeros, ``On the finite-time behavior of suboptimal linear model predictive control,'' in {\em 2023 62nd IEEE Conference on Decision and Control (CDC)}, pp.~5053--5058, IEEE, 2023.

\bibitem{srikanthan2023augmented}
A.~Srikanthan, V.~Kumar, and N.~Matni, ``Augmented lagrangian methods as layered control architectures,'' {\em arXiv preprint arXiv:2311.06404}, 2023.

\bibitem{boyd2011distributed}
S.~Boyd, N.~Parikh, E.~Chu, B.~Peleato, J.~Eckstein, {\em et~al.}, ``Distributed optimization and statistical learning via the alternating direction method of multipliers,'' {\em Foundations and Trends{\textregistered} in Machine learning}, vol.~3, no.~1, pp.~1--122, 2011.

\bibitem{schulze2021closed}
M.~Schulze~Darup and G.~Book, ``On closed-loop dynamics of {ADMM}-based {MPC},'' {\em Recent Advances in Model Predictive Control: Theory, Algorithms, and Applications}, pp.~107--134, 2021.

\bibitem{liao2021analysis}
D.~Liao-McPherson, T.~Skibik, J.~Leung, I.~Kolmanovsky, and M.~M. Nicotra, ``An analysis of closed-loop stability for linear model predictive control based on time-distributed optimization,'' {\em IEEE Transactions on Automatic Control}, vol.~67, no.~5, pp.~2618--2625, 2021.

\bibitem{leung2021computable}
J.~Leung, D.~Liao-McPherson, and I.~V. Kolmanovsky, ``A computable plant-optimizer region of attraction estimate for time-distributed linear model predictive control,'' in {\em 2021 American Control Conference (ACC)}, pp.~3384--3391, IEEE, 2021.

\bibitem{matni2016theory}
N.~Matni and J.~C. Doyle, ``A theory of dynamics, control and optimization in layered architectures,'' in {\em 2016 American Control Conference (ACC)}, pp.~2886--2893, IEEE, 2016.

\bibitem{matni2024towards}
N.~Matni, A.~D. Ames, and J.~C. Doyle, ``Towards a theory of control architecture: A quantitative framework for layered multi-rate control,'' {\em arXiv preprint arXiv:2401.15185}, 2024.

\bibitem{o2013splitting}
B.~O'Donoghue, G.~Stathopoulos, and S.~Boyd, ``A splitting method for optimal control,'' {\em IEEE Transactions on Control Systems Technology}, vol.~21, no.~6, pp.~2432--2442, 2013.

\bibitem{nishihara2015general}
R.~Nishihara, L.~Lessard, B.~Recht, A.~Packard, and M.~Jordan, ``A general analysis of the convergence of {ADMM},'' in {\em International conference on machine learning}, pp.~343--352, PMLR, 2015.

\bibitem{east2018admm}
S.~East and M.~Cannon, ``{ADMM} for {MPC} with state and input constraints, and input nonlinearity,'' in {\em 2018 annual American control conference (ACC)}, pp.~4514--4519, IEEE, 2018.

\bibitem{rey2020admm}
F.~Rey, P.~Hokayem, and J.~Lygeros, ``{ADMM} for exploiting structure in {MPC} problems,'' {\em IEEE Transactions on Automatic Control}, vol.~66, no.~5, pp.~2076--2086, 2020.

\bibitem{gracia2024efficient}
V.~Gracia, P.~Krupa, D.~Limon, and T.~Alamo, ``Efficient implementation of {MPC} for tracking using {ADMM} by decoupling its semi-banded structure,'' {\em arXiv preprint arXiv:2402.09912}, 2024.

\bibitem{krupa2024sparse}
P.~Krupa, R.~Jaouani, D.~Limon, and T.~Alamo, ``A sparse {ADMM}-based solver for linear {MPC} subject to terminal quadratic constraint,'' {\em IEEE Transactions on Control Systems Technology}, 2024.

\bibitem{sindhwani2017sequential}
V.~Sindhwani, R.~Roelofs, and M.~Kalakrishnan, ``Sequential operator splitting for constrained nonlinear optimal control,'' in {\em 2017 American Control Conference (ACC)}, pp.~4864--4871, IEEE, 2017.

\bibitem{bertsekas1997nonlinear}
D.~P. Bertsekas, ``Nonlinear programming,'' {\em Journal of the Operational Research Society}, vol.~48, no.~3, pp.~334--334, 1997.

\bibitem{robinson1980strongly}
S.~M. Robinson, ``Strongly regular generalized equations,'' {\em Mathematics of Operations Research}, vol.~5, no.~1, pp.~43--62, 1980.

\bibitem{dontchev2013euler}
A.~L. Dontchev, M.~Krastanov, R.~T. Rockafellar, and V.~M. Veliov, ``An euler--newton continuation method for tracking solution trajectories of parametric variational inequalities,'' {\em SIAM Journal on Control and Optimization}, vol.~51, no.~3, pp.~1823--1840, 2013.

\bibitem{limon2006stability}
D.~Lim{\'o}n, T.~Alamo, F.~Salas, and E.~F. Camacho, ``On the stability of constrained {MPC} without terminal constraint,'' {\em IEEE transactions on automatic control}, vol.~51, no.~5, pp.~832--836, 2006.

\bibitem{jiang2004nonlinear}
Z.-P. Jiang, Y.~Lin, and Y.~Wang, ``Nonlinear small-gain theorems for discrete-time feedback systems and applications,'' {\em Automatica}, vol.~40, no.~12, pp.~2129--2136, 2004.

\bibitem{bemporad2002explicit}
A.~Bemporad, M.~Morari, V.~Dua, and E.~N. Pistikopoulos, ``The explicit linear quadratic regulator for constrained systems,'' {\em Automatica}, vol.~38, no.~1, pp.~3--20, 2002.

\bibitem{chen2024robust}
S.~Chen, V.~M. Preciado, M.~Morari, and N.~Matni, ``Robust model predictive control with polytopic model uncertainty through system level synthesis,'' {\em Automatica}, vol.~162, p.~111431, 2024.

\end{thebibliography}

\appendix

\section{Derivation of Proposition~\ref{prop:admm-convergence-thm6}}\label{app:prop1}
We bisect on values of $\alpha \in (0, 2)$ to find its maximum value for which the linear matrix inequality (LMI) in \cite[Theorem 6]{nishihara2015general} is satisfied, ensuring the algorithm's feasibility for our problem parameters. Once $\alpha$ is determined, we provide the details of the derivation of Proposition~\ref{prop:admm-convergence-thm6} here. Following~\cite{nishihara2015general}, we rewrite the iterations of the optimizer as a discrete-time dynamical system on state sequence $\xi^k$, input sequence $\nu^k$, and output sequences $w_1^{k}$ and $w_2^{k}$ satisfying recursions as
\begin{subequations}\label{eq:admm-dyn-recursion}
    \begin{align}
        \xi^{k+1} = (\hat{A} \kron I_s) \xi^k + (\hat{B} \kron I_s)\nu^k \label{eq:admm-state-recursion} \\
        w_1^{k} = (\hat{C}_1 \kron I_s) \xi^k + (\hat{D}_1 \kron I_s) \nu^k \label{eq:admm-w1-recursion} \\
        w_2^{k} = (\hat{C}_2 \kron I_s) \xi^k + (\hat{D}_2 \kron I_s) \nu^k \label{eq:admm-w2-recursion}
    \end{align}
\end{subequations}
where $\xi^k = \begin{bmatrix}
    \boldsymbol{r}^k \\ \boldsymbol{v}^k
\end{bmatrix}$, and the corresponding matrices are
\begin{equation*}
    \hat{A} = \begin{bmatrix}
        1 & 1-\alpha \\
        0 & 0
    \end{bmatrix}, \quad \hat{B} = \begin{bmatrix}
        -\alpha & -1 \\
        0 & 1
    \end{bmatrix}.
\end{equation*}

Next, for $J, g$ we define the sequences $(w_1^k)$ and $(w^k_2)$ via
\begin{equation*}
    w^k_1 = \begin{bmatrix}
\boldsymbol{r}^{k+1} \\
\beta^{k+1}
\end{bmatrix} \quad \text{and} \quad w^k_2 = \begin{bmatrix}
\boldsymbol{u}^{k+1} \\
\gamma^{k+1}
\end{bmatrix}.
\end{equation*}
such that there exist sequences $(\beta^k)$ and $(\gamma^k)$ with $\beta^k = \nabla (\rho^{-1} f)(\mathbf{r}^k)$ and $\gamma^k \in \partial (\rho^{-1}g)(\mathbf{z}^k)$ where $(\nu^k) = (\beta^{k+1}, \gamma^{k+1})$. Then the sequences $(\xi^k)$, $(\nu^k)$, $(w^k_1)$, and $(w^k_2)$ satisfy~\eqref{eq:admm-w1-recursion} and~\eqref{eq:admm-w2-recursion} with matrices
\begin{align*}
    \hat{C}^1 &= \begin{bmatrix}
        1 & -1 \\
        0 & 0
    \end{bmatrix}, \quad \hat{C}^2 = \begin{bmatrix}
        1 & 1-\alpha \\
        0 & 0
    \end{bmatrix}, \\ 
    \hat{D}^1 &= \begin{bmatrix}
        -1 & 0 \\
        1 & 0
    \end{bmatrix}, \quad \hat{D}^2 = \begin{bmatrix}
        -\alpha & -1 \\
        0 & 1
    \end{bmatrix}.
\end{align*}

For sufficiently large $\kappa = L/p$ and
$\rho = (p L)^{\frac{1}{2}} \rho_0$ where $\rho_0 = \kappa^\epsilon > 0$, we have that Proposition~\ref{prop:admm-convergence-thm6} holds satisfying the LMI we describe below with $\bar{F} = \begin{bmatrix}
    1 & 1 - \alpha \\
    1 - \alpha & 1
\end{bmatrix}$ and convergence rate $\tau = 1 - \frac{\alpha}{2 \kappa^{0.5 + |\epsilon|}}$. 

Let $J$ be $p-$strongly convex and $L-$smooth function, and fix the update rate to $\rho=\sqrt{pL}\rho_0$ for some $\rho_0 = \kappa^\epsilon >0$. For a fixed $\tau = 1 - \frac{\alpha}{2 \kappa^{0.5 + |\epsilon|}} \in (0,1)$, if there exists a $\bar{F}\in \mathbb{S}^2_{++}$ and constants $\lambda^1,\lambda^2>0$, such that the following LMI is satisfied
    \begin{equation*}
        \begin{bmatrix}
            \hat{A}^\top \bar{F}\hat{A} - \tau^2\bar{F} & \hat{A}^\top \bar{F}\hat{B}\\
            \hat{B}^\top \bar{F}\hat{A}& \hat{B}^\top \bar{F}\hat{B}
        \end{bmatrix}+
        \begin{bmatrix}
            \hat{C}^1 & \hat{D}^1\\
            \hat{C}^2 & \hat{D}^2
        \end{bmatrix}^\top
        \begin{bmatrix}
            \lambda^1M^1& 0\\
            0& \lambda^2M^2
        \end{bmatrix}
        \begin{bmatrix}
            \hat{C}^1 & \hat{D}^1\\
            \hat{C}^2 & \hat{D}^2
        \end{bmatrix}\preceq 0,
    \end{equation*}
    where $\hat{A}, \hat{B}, \hat{C}^1, \hat{C}^2$ are defined above and $M^1$ and $M^2$ are given by
    \begin{equation*}
    M^1 = \begin{bmatrix}
        -2\kappa^{-2\epsilon} & \kappa^{-1/2-\epsilon} + \kappa^{1/2-\epsilon} \\
        \kappa^{-1/2-\epsilon} + \kappa^{1/2-\epsilon} & -2
    \end{bmatrix}, \quad M^2 = \begin{bmatrix}
        0 & 1 \\
        1 & 0
    \end{bmatrix}.
\end{equation*}
Then for all $t\geq0$, and $x_t$ generated by our ADMM algorithm~\eqref{eq:linear_ocp_admm} satisfies
    \begin{equation*}
    \left \|{\boldsymbol{\varphi}}_t^\ell - {\boldsymbol{\varphi}}^\star(x_{t}) \right \|_{F} \leq \tau^{\ell}  \left\|{\boldsymbol{\varphi}}_{t-1}^\ell -  {\boldsymbol{\varphi}}^\star(x_{t}) \right \|_{F},
    \end{equation*}
where $F = \bar{F} \kron I_{2s}$.

\section{Detailed steps in the proof of Theorem~\ref{thm:asympt-stable-suboptimal-cl}}\label{app:thm-4}
To show that  $\|\boldsymbol{e}^{\ell}_{t+1}\|\leq r_e$, consider the following inequality from Theorem 3
    \begin{align*}
         \| \boldsymbol{e}^{\ell}_{t+1} \|_F &\leq \tau^\ell \|\boldsymbol{e}^{\ell}_{t}\|_F + \tau^\ell L_1{\|F^{\frac{1}{2}}\|}\|\Delta x_{t+1}\| \\
         &\leq \tau^\ell \|F^{\frac{1}{2}}\| \|\boldsymbol{e}^{\ell}_{t}\| + \tau^\ell L_1{\|F^{\frac{1}{2}}\|}\|\Delta x_{t+1}\|
         \end{align*}
    where the second inequality follows from the submultiplicativity property of norms.
    It also holds from (13) that

\begin{equation*}
    x_{t+1} = \underbrace{Ax_t + \bar{B}{\boldsymbol{\varphi}}^\star(x_t)}_{f(x_t)} + \bar{B}\boldsymbol{e}_t^\ell.
\end{equation*}
Subtracting $x_t$ from both sides, it follows that
\begin{equation*}
    \Delta x_{t+1} = x_{t+1} - x_t = (A - I_n)x_t + \bar{B}{\boldsymbol{\varphi}}^\star(x_t)+ \bar{B}\boldsymbol{e}_t^\ell.
\end{equation*}
Taking the norm of both sides
         \begin{align*} \|\Delta x_{t+1}\| \leq  \|(A-I_n)x_t\| + \|\bar{B}{\boldsymbol{\varphi}}^\star(x_t)\| + \|\bar{B}\boldsymbol{e}^{\ell}_t\|.
    \end{align*}
    Combining this with the bound on $\|\boldsymbol{e}_{t+1}^\ell\|_F$ from above, and using submultiplicativity property of norms and Corollary 1, it holds that

\begin{align*}
    \sqrt{\lambda^-(F)} \|\boldsymbol{e}_{t+1}^\ell\| \leq \|\boldsymbol{e}_{t+1}^\ell\|_F &\leq  \tau^\ell \|F^{\frac{1}{2}}\| \|\boldsymbol{e}^{\ell}_{t}\| + \tau^\ell L_1 \|F^{\frac{1}{2}}\|\left(\|(A-I_n)x_t\| + \|\bar{B}{\boldsymbol{\varphi}}^\star(x_t)\| + \|\bar{B}\boldsymbol{e}^{\ell}_t\|\right)\\
    &\leq \tau^\ell \|F^{\frac{1}{2}}\| \|\boldsymbol{e}^{\ell}_{t}\| + \tau^\ell L_1 \|F^{\frac{1}{2}}\| \left(\|(A-I_n)x_t\|+ L_1\|\bar{B}\|\|x_t\| + \|\bar{B}\|\|\boldsymbol{e}_t^\ell\|\right)\\
    &\leq \tau^\ell \|F^{\frac{1}{2}}\|\left(1+L_1 \|\bar{B}\|\right) \|\boldsymbol{e}_t^\ell\| + \tau^\ell L_1 \|F^{\frac{1}{2}}\| \left(\|(A-I_n)x_t\|+ L_1\|\bar{B}\|\|x_t\|\|\right)
\end{align*}
Denoting $\kappa_F$ to be the condition number of $F$, we then divide both sides of the inequality by $\sqrt{\lambda^-(F)}$, and note that $\|x_t\| = \|P^{-\frac{1}{2}}P^{\frac{1}{2}}x_t\| \leq \|P^{-\frac{1}{2}}\| \|x_t\|_P$; it then follows that
    \begin{align*}
        \| \boldsymbol{e}^{\ell}_{t+1} \| &\leq \tau^\ell \underbrace{\sqrt{\kappa_F}\left(1+L_1 \|\bar{B}\|\right)}_{:=\omega_1} \|\boldsymbol{e}^{\ell}_t\| + \tau^\ell \underbrace{L_1\sqrt{\kappa_F} \left(\|(A-I_n)P^{-\frac{1}{2}}\| + L_1 \|\bar{B}\|\|P^{-\frac{1}{2}}\|\right)}_{:=\omega_2} \|x_t\|_P\\
        &= \tau^\ell\omega_1 \|\boldsymbol{e}^{\ell}_t\| + \tau^\ell \omega_2 \|x_t\|_P
    \end{align*}
    Then, using the definition of $\Omega$ and the fact that $\|x_t\|_P \leq \psi(x_t)$ by Lemma 1, it holds that
    \begin{equation*}
        \| \boldsymbol{e}^{\ell}_{t+1} \| \leq \tau^\ell\left(\frac{\omega_1(1-\beta)r_N}{\delta\|\bar{B}\|} + \omega_2 r_N\right) \leq r_e
    \end{equation*}
    if 
    \begin{equation*}
        \ell \geq  -\frac{{\log}{\left(\omega_1 + \frac{\omega_2\delta\|B\|}{1-\beta}\right)}}{\log{\tau}}
    \end{equation*}
    Thus, the forward invariance of $\Omega$ is shown and therefore $\{\boldsymbol{e}^\ell_i\}_{i=0}^t \in \mathcal{E}$ satisfying the condition in Theorem 2.

\section{A summary of our main result in Theorem~\ref{thm:asympt-stable-suboptimal-cl}}\label{app:thm-4-summary}
Since Theorem 4 is our main result and encapsulates all of the bounds and constants used in the previous lemmas and theorems, we provide an explanation below on how these can be computed in practice.

\textbf{Definition of the ROA set $\Omega$:} The forward invariant region of attraction set estimate $\Omega$ consists of two parts: ensuring that the state is in $\Gamma_N$ and that the error due to suboptimality satisfies $\|\boldsymbol{e}^\ell\| \leq \frac{(1-\beta)\sqrt{N d+c}}{\delta \|\bar{B}\|}$. The definition of $\Gamma_N$ comes from \cite{limon2006stability} directly and it is not the focus of our paper to derive an explicitly computable alterative to $\Gamma_N$. In fact, any exponentially stable MPC design (i.e. not necessarily without terminal constraints) shall also work for our analysis. The constant $d$ depends directly on $c$ as we show before (13) and $c$ can be easily computed by solving an LMI. The constant $\beta$ depends on $\delta$ which can be explicitly calculated except for the Lipschitz constant $L_1$ in the general case. Below we explain how one can get an estimate for $L_1$.

\textbf{The expression for $\ell^\star$:}
$\ell^\star$ depends on a number of directly available parameters such as the system matrices $A,B$, the ARE solution $P$, the MPC horizon length $N$ and the cost matrices $Q$ and $R$. It also depends on ADMM parameters such as $F$ and $\tau$ which can be calculated as explained in detail in \cite{nishihara2015general}. Finally, it depends on $L_1$, which we explain how to compute below.

\textbf{The Lipschitz constant $L_1$:}

In the absence of state constraints, it is already known \cite[Corollary 2]{liao2020time} that the optimal primal solution of the OCP (5) is Lipschitz continuous with respect to the parameter $\xi$, and an explicit expression is available. In our notation, this is given as follows: for any $\xi,\bar{\xi}\in \mathbb{R}^n$

\begin{equation*}
    \|\boldsymbol{r}(\xi) - \boldsymbol{r}(\bar{\xi})\| \leq \|H^{-\frac{1}{2}}\|\|H^{-\frac{1}{2}}G\|\|\xi-\bar{\xi}\|
\end{equation*}

From the necessary KKT condition (9), it follows that (we ignore the artificial initial state constraint and the dual variable $\boldsymbol{\lambda}$ since there are no state constraints)

\begin{equation*}
    \begin{bmatrix}
        H\boldsymbol{u} + G\xi + \boldsymbol{y} = 0\\
        -\boldsymbol{y} + \mathcal{N}_{\mathcal{U}^N}(\boldsymbol{r}) \ni 0 \\
        \boldsymbol{u} - \boldsymbol{r} =0
    \end{bmatrix}.
\end{equation*}

Combining this with the optimal primal solution expression, it can be deduced that

    \begin{equation*}
    \|\boldsymbol{y}(\xi) - \boldsymbol{y}(\bar{\xi})\| \leq \left(\|H\|\|H^{-\frac{1}{2}}\|\|H^{-\frac{1}{2}}G\|+\|G\|\right)\|\xi-\bar{\xi}\|.
\end{equation*}

Combining the two  and noting that $ \boldsymbol{\varphi^\star}(\xi)= \left[\boldsymbol{r}^{\star\top}(\xi), \boldsymbol{y}^{\star\top}(\xi)\right]^\top$, it follows that one can take 
\\
$L_1 = \|H^{-\frac{1}{2}}\|\|H^{-\frac{1}{2}}G\|\left(\|H\|+1\right) + \|G\|$.

When state constraints are present, the result from \cite{liao2020time} no longer holds. So one can assume that the system moves away from the state constraints quickly (as is the case in our practical example),
or use other means of deriving such a bound. For example, one can leverage the piecewise affine nature of the optimal MPC policy \cite{bemporad2002explicit} to show the Lipschitz continuity of the primal variable. The result for the dual variable will follow in the exact same way. 

\section{Additional numerical comparisons}\label{app:sim}

\begin{figure}[t]
    \centering
    \begin{subfigure}{0.49\linewidth}
         \includegraphics[width=\textwidth]{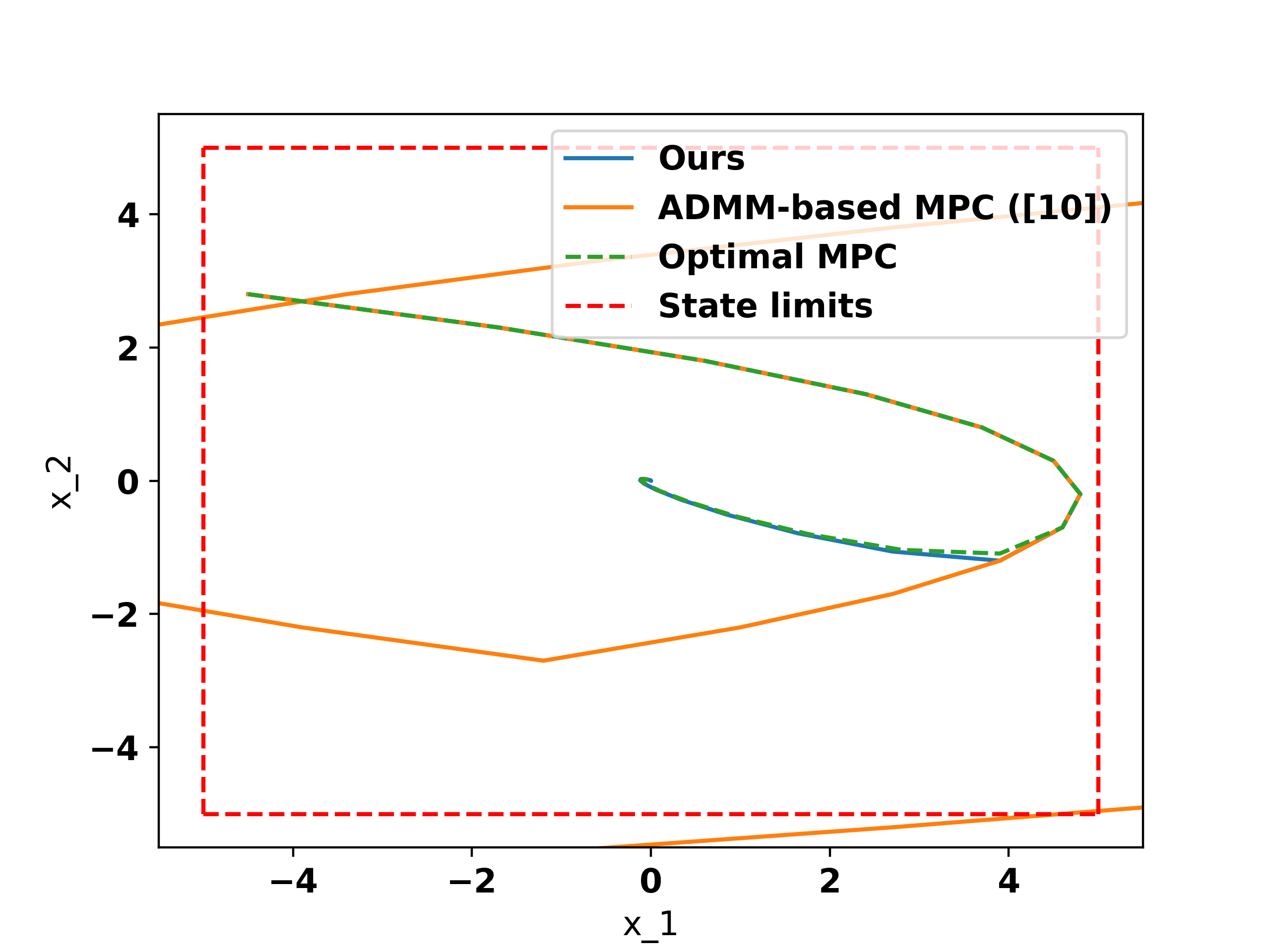}
         \caption{State trajectories\label{fig:states2}}
    \end{subfigure}
    \begin{subfigure}{0.49\linewidth}
         \includegraphics[width=\linewidth]{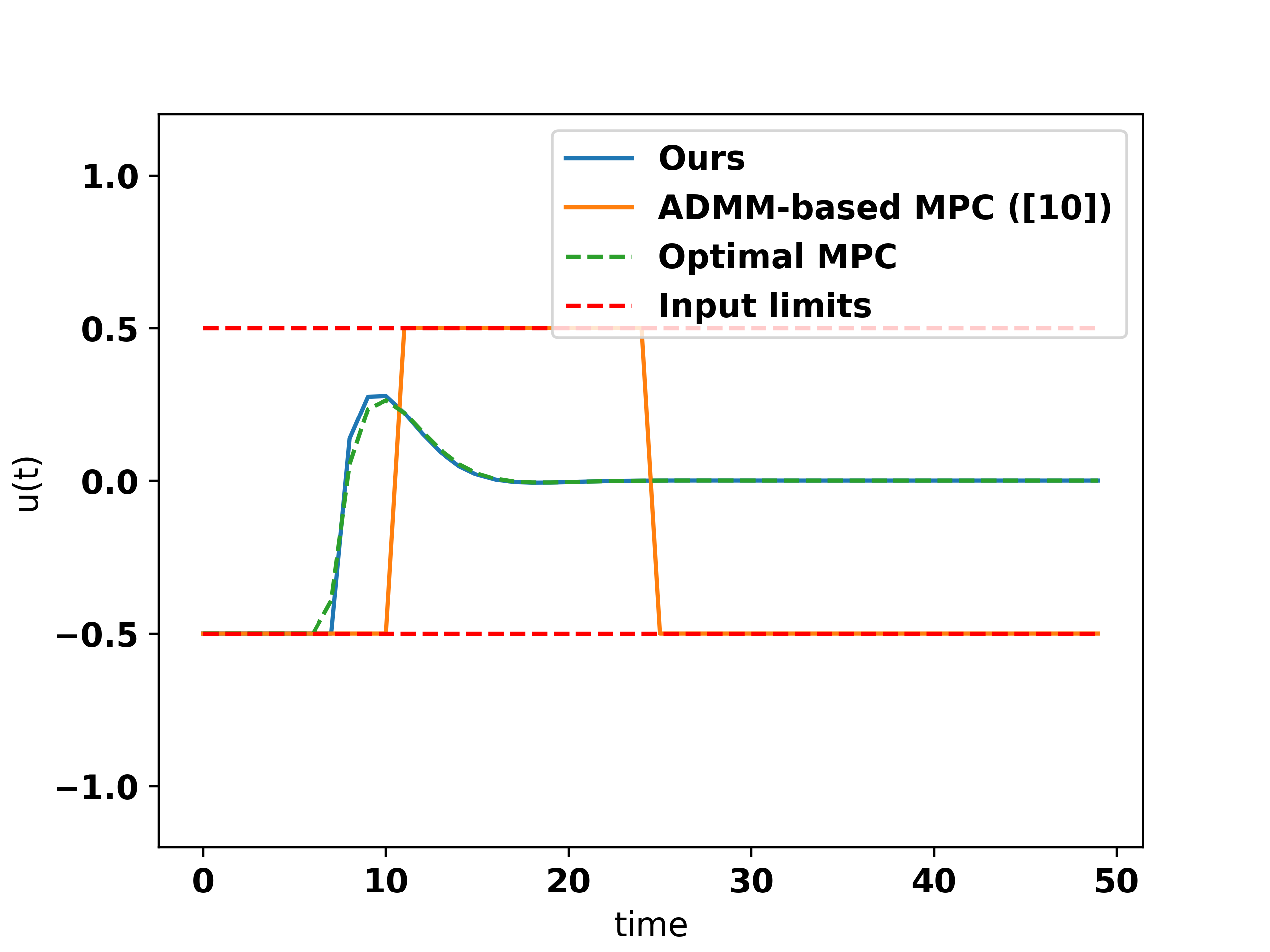}
         \caption{Input trajectories\label{fig:inputs2}}
    \end{subfigure}
    \caption{We provide a comparison of state trajectories in (a) and inputs in (b) computed by our approach and the algorithm from~\cite{schulze2021closed} for the double integrator example in Section V of our paper. The red dotted lines denote the boundaries of state and input constraints. The blue lines are from our approach, the green dotted lines are from the optimal MPC and the orange lines are from~\cite{schulze2021closed}. We run $\ell = 30$ ADMM iterations for our approach.}   
    \label{fig:experiment1}
\end{figure}

To further bolster our theoretical contributions, we provide numerical simulations in Figures~\ref{fig:experiment1} and~\ref{fig:experiment2} comparing our approach with the optimal MPC and a suboptimal MPC scheme from~\cite{schulze2021closed} for linear optimizer-system closed-loop dynamics. We demonstrate that our theoretical analysis, may be conservative, but is nonetheless helpful to design an ADMM-based suboptimal MPC algorithm for nonlinear optimizer-system closed-loop dynamics. In Figure~\ref{fig:experiment1}, we plot the closed-loop state and input trajectories for the double integrator example in Section~\ref{sec:num-sim} under our proposed policy with $\ell=30$, the algorithm from \cite{schulze2021closed} and optimal MPC, all starting from the same initial state $[-4;2.8]$, near the state constraint boundary. Similarly, in Figure~\ref{fig:experiment2}, we compare our approach using the numerical example from~\cite{schulze2021closed}. It can be seen, that our method renders the closed-loop stable, together with the optimal MPC, even though the exact local ROA bound and $\ell^\star$ is not calculated. The method of~\cite{schulze2021closed} that studied the closed-loop ADMM-MPC dynamics only when none of the constraints were active becomes unstable.

\begin{figure}[t]
    \centering
    \begin{subfigure}{0.49\linewidth}
         \includegraphics[width=\textwidth]{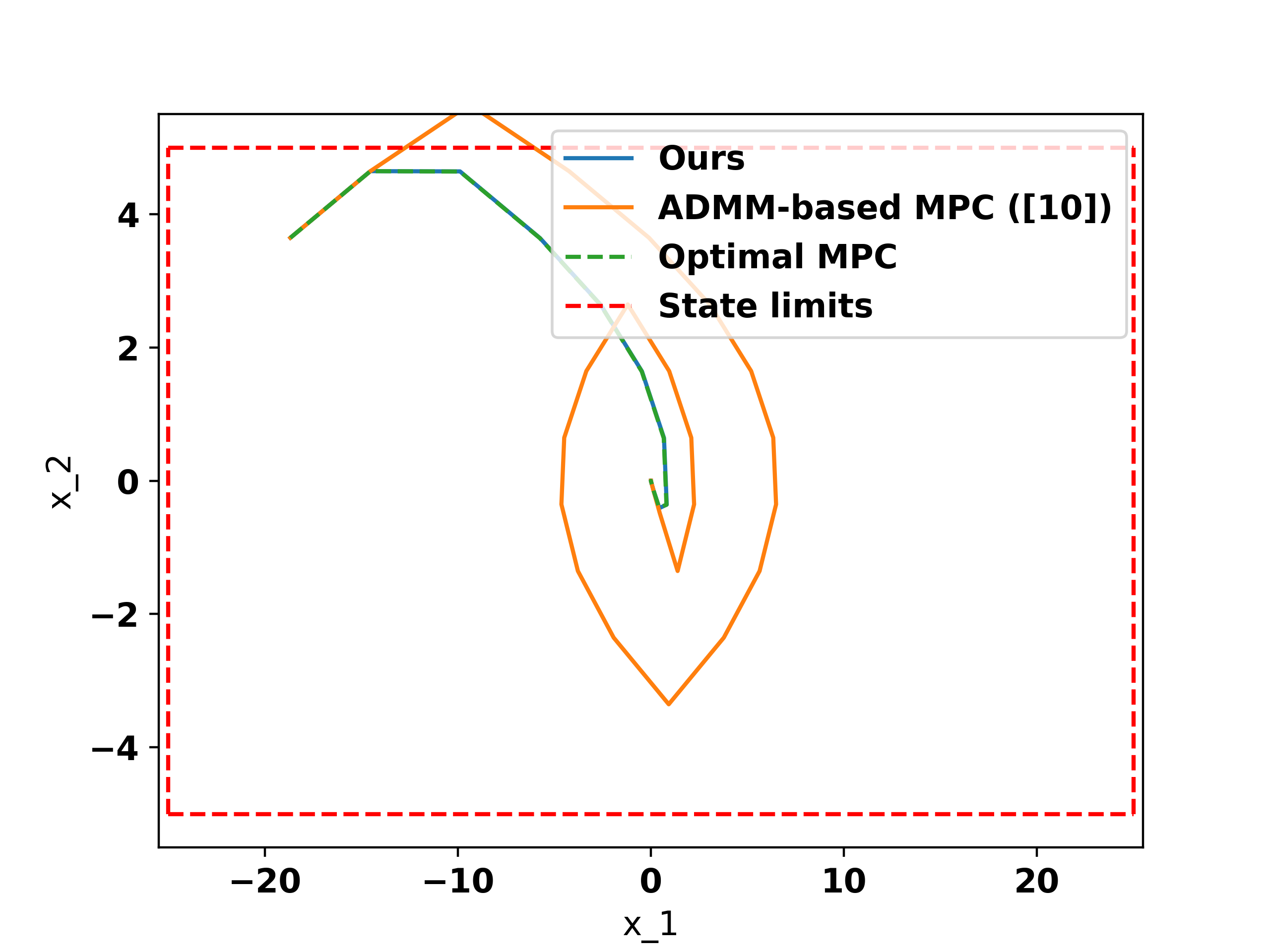}
    \caption{State trajectories\label{fig:states3}}
\end{subfigure}\begin{subfigure}{0.49\linewidth}
         \includegraphics[width=\linewidth]{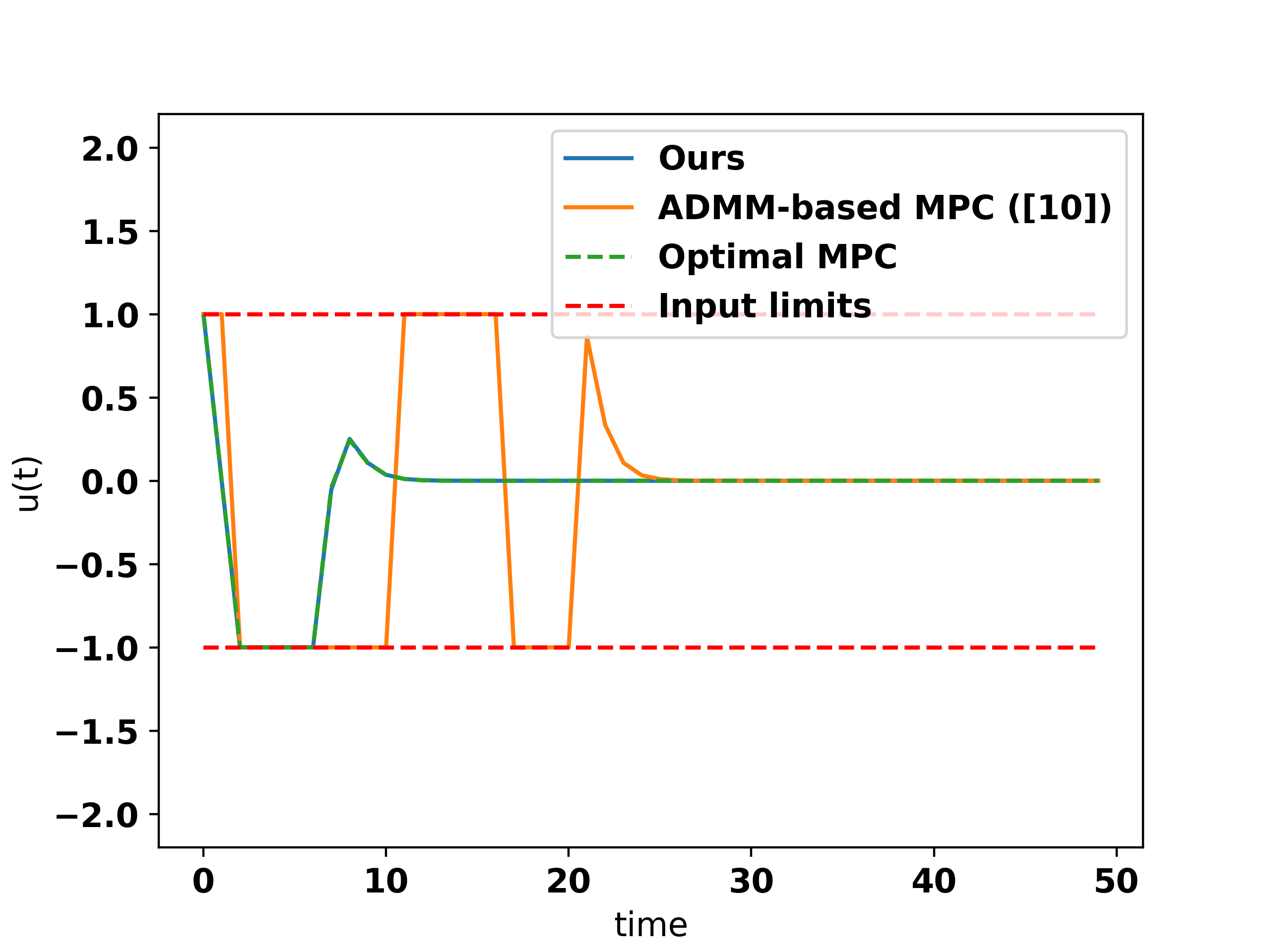}
     \caption{Input trajectories\label{fig:inputs3}}
    \end{subfigure}
    \caption{We provide a comparison of state trajectories in (a) and inputs in (b) computed by our approach and the algorithm from~\cite{schulze2021closed} for the double integrator example from~\cite{schulze2021closed} with $B = [0.5; 1]$. The red dotted lines denote the boundaries of state and input constraints. The blue lines are from our approach, the green dotted lines are from the optimal MPC and the orange lines are from~\cite{schulze2021closed}. We run $\ell = 30$ ADMM iterations for our approach.}   
    \label{fig:experiment2}
\end{figure}

\end{document}